\documentclass[article,11pt]{amsart}
\usepackage[top=25mm,bottom=25mm,left=25mm,right=25mm]{geometry}
\usepackage{graphicx,tikz}
\usetikzlibrary{decorations.pathreplacing}    
\usepackage[outline]{contour}

\usepackage{float}

\usepackage{mathtools,hyperref,amsthm,amssymb,cases,wasysym,amsmath,pgffor}
\usepackage{color}

\newtheorem{lemma}{Lemma}[section]

\newtheorem{proposition}{Proposition}[section]
\newtheorem{theorem}{Theorem}[section]
\newtheorem{corollary}{Corollary}[section]
\newtheorem{definition}{Definition}[section]

\theoremstyle{definition}

\makeatletter
\def\section{\@startsection{section}{1}%
\z@{1\linespacing\@plus\linespacing}{1\linespacing}%
{\bf\centering}}
\def\subsection{\@startsection{subsection}{0}%
\z@{\linespacing\@plus\linespacing}{\linespacing}%
{\bf}}
\makeatother

\makeatletter
\@addtoreset{equation}{section}
\makeatother

\DeclareMathOperator{\conv}{conv}

\usetikzlibrary{calc}     %obliczanie współrzędnych 
\newcommand{\sz}[1]{
	\begin{scope}[shift={(#1)}]
		\coordinate (0) at (0,0);
		\foreach \x in {1,2,3,...,6}{
			\path (0) ++(\x *60:1cm) coordinate (\x );
		}
		\filldraw[fill=lightgray,draw=black] (1) -- (2) -- (3) -- (4) -- (5) -- (6)  --cycle;
		\foreach \x in {1,2,3,...,6}{
			\filldraw [black] (\x) circle (1.5pt);
		}
	\end{scope}
}
\newcommand{\szb}[1]{
	\begin{scope}[shift={(#1)}]
		\coordinate (0) at (0,0);
		\foreach \x in {1,2,3,...,6}{
			\path (0) ++(\x *60:1cm) coordinate (\x );
		}
		\filldraw[fill=brown,draw=black] (1) -- (2) -- (3) -- (4) -- (5) -- (6)  --cycle;
		\foreach \x in {1,2,3,...,6}{
			\filldraw [black] (\x) circle (1.5pt);
		}
	\end{scope}
}
\newcommand{\ds}[1]{
	\begin{scope}[shift={(#1)}]
		\coordinate (00) at (0,0);
		\foreach \x in {0,1,2,3,...,9}{
			\path (00) ++(\x *36:1cm) coordinate (\x );
		}
		\filldraw[fill=lightgray,draw=black] (0) -- (1) -- (2) -- (3) -- (4) -- (5) -- (6)  -- (7)  -- (8) -- (9)--cycle;
		\foreach \x in {0,1,2,3,...,9}{
			\filldraw [black] (\x) circle (1.5pt);
		}
	\end{scope}
}
\newcommand{\dsone}[1]{
	\begin{scope}[shift={(#1)}]
		\coordinate (00) at (0,0);
		\foreach \x in {0,1,2,3,...,9}{
			\path (00) ++(\x *36:1cm) coordinate (\x );
		}
		\filldraw[fill=lightgray,draw=black] (0) -- (1) -- (2) -- (3) -- (4) -- (5) -- (6)  -- (7)  -- (8) -- (9)--cycle;
		\foreach \x in {0,1,2,3,...,9}{
			\filldraw [black] (\x) circle (1.5pt);
		}
		\node (00) {1};
	\end{scope}
}
\newcommand{\dsoneb}[1]{
	\begin{scope}[shift={(#1)}]
		\coordinate (00) at (0,0);
		\foreach \x in {0,1,2,3,...,9}{
			\path (00) ++(\x *36:1cm) coordinate (\x );
		}
		\filldraw[fill=brown,draw=black] (0) -- (1) -- (2) -- (3) -- (4) -- (5) -- (6)  -- (7)  -- (8) -- (9)--cycle;
		\foreach \x in {0,1,2,3,...,9}{
			\filldraw [black] (\x) circle (1.5pt);
		}
	\end{scope}
}
\newcommand{\dstwo}[1]{
	\begin{scope}[shift={(#1)}]
		\coordinate (00) at (0,0);
		\foreach \x in {0,1,2,3,...,9}{
			\path (00) ++(\x *36:1cm) coordinate (\x );
		}
		\filldraw[fill=lightgray,draw=black] (0) -- (1) -- (2) -- (3) -- (4) -- (5) -- (6)  -- (7)  -- (8) -- (9)--cycle;
		\foreach \x in {0,1,2,3,...,9}{
			\filldraw [black] (\x) circle (1.5pt);
		}
		\node at(-0.4,0) {2};
	\end{scope}
}
\newcommand{\kt}[1]{#1*36}
\newcommand{\kat}[1]{#1*60}
\mathtoolsset{showonlyrefs,mathic}

\newcommand{\cH}{\mathcal{H}}

\newcommand{\cA}{\mathcal{A}}

\newcommand{\cK}{\mathcal{K}}

\begin{document}
\title[Good labeling property of simple nested fractals]
{Good labeling property of simple nested fractals}
\author{Miko{\l}aj Nieradko and Mariusz Olszewski}

\address{M. Nieradko and M. Olszewski \\ Faculty of Pure and Applied Mathematics, Wroc{\l}aw University of Science and Technology, Wyb. Wyspia\'nskiego 27, 50-370 Wroc{\l}aw, Poland}
\email{mikolaj.nieradko@student.pwr.edu.pl, mariusz.olszewski@pwr.edu.pl}

\begin{abstract}
{We show various criteria to verify if a given nested fractal has a good labeling property, inter alia we present a characterization of GLP for fractals with an odd number of essential fixed points. \\
We show a convenient reduction of area to be investigated in verification of~GLP and give examples that further reduction is impossible.\\
We prove that if a number of essential fixed points is a power of two, then a fractal must have GLP and that there are no values other than primes or powers of two guaranteeing GLP. For all other numbers of essential fixed points we are able to construct examples having and other not having GLP.}


\bigskip
\noindent
2020 {\it MS Classification}: Primary: 52C45, 28A80, 05B25.
\end{abstract}

%\footnotetext{Research was supported by the National Science Center, Poland, grant no. 2015/17/B/ST1/01233 and by the Alexander von Humboldt Foundation, Germany.}

\maketitle

\baselineskip 0.5 cm

\bigskip\bigskip

\section{Introduction}

Fractal sets are investigated by mathematicians for more than a hundred years. In 1915 Sierpi\'nski constructed a {\em curve, which every point is a branching point} later known as the Sierpi\'nski gasket. Benoit Mandelbrot in the second half of 20th century developed a theory of self-similar sets, introduced a notion of {\em a fractal} and contributed to many fields of science. He found self-similarity in financial markets, lungs structure, leaf shapes, shorelines (asking famous question: {\em How long is the coast of Britain?} \cite{bib:Mandelbrot}), and even a distribution of galaxy clusters.

In the late 20th century fractal sets became an object or rising interest of mathematicians specializing in analysis and probability theory, who found in them an interesting domain for operators and a state space for stochastic processes. In 1987-88 Barlow and Perkins \cite{bib:BP}, Goldstein \cite{bib:Goldstein} and Kusuoka \cite{bib:Kus} independently constructed the Brownian motion on the Sierpi\'nski gasket. In 1989 the existence of an analogous process on the Sierpi\'nski carpet was proved \cite{bib:BarBas} and in 1990 Lindstr\o m carried out a construction of the Brownian motion on a wide class of nested fractals. The corresponding Dirichlet form was given by Fukushima \cite{bib:Fuk1}. Kumagai \cite{bib:Kum} provided estimates of a free diffusion on nested fractals. Later, even more general classes like P.C.F sets and affine nested fractals were considered \cite{bib:Kig1,bib:FHK}. We refer to \cite{bib:Bar} as the general source of information about processes on fractals.

Stochastic processes on fractals are a useful tool in physical applications. Pietruska-Pa\l uba \cite{bib:KPP-PTRF} and Kaleta with Pietruska-Pa\l uba \cite{bib:KK-KPP,bib:KK-KPP2} proved existence and analyzed properties of the integrated density of states (IDS) for Schr\"odinger perturbations of a subordinate Brownian motion on Sierpi\'nski gasket. Later Balsam, Kaleta, Olszewski and Pietruska-Pa\l uba generalized a construction of IDS on nested fractals having the good labeling property (GLP). This property is crucial in construction of the reflected Brownian motion on nested fractals \cite{bib:KOPP,bib:O}, which is a relevant tool in proving an existence of IDS and analyzing its asymptotics \cite{bib:BKOPP1,bib:BKOPP2,bib:Balsam}.

GLP allows to construct a continuous projection from an unbounded fractal onto single complex of a given size. It can be also used to define periodic functions on nested fractals. Such objects on Sierpi\'nski gasket and its higher dimension analogues were examined by Ruan and Strichartz \cite{bib:RS, bib:Str}.

The subclass of nested fractals having GLP is very broad. In \cite{bib:KOPP} it was proved that fractals with prime number of essential fixed points always have GLP. It was also shown that if a fractal does not have inessential fixed points, it must have GLP. Finally, there was given characterization of GLP for fractals with even number of essential fixed points.

In this paper we thoroughly investigate the good labeling property and generalize the results from \cite{bib:KOPP}. In chapter 2 we give basic information about the nested fractals, the good labeling property and introduce the notation. In chapter 3 we present new results. Theorem \ref{thm:odd} is a counterpart of Theorem 3.4 in \cite{bib:KOPP} and gives characterization of GLP for fractals with odd number of essential fixed points. In case of odd number of essential fixed points Theorem 3.3 in \cite{bib:KOPP} can be read as a simple conclusion from our Theorem \ref{thm:odd}.

Theorems \ref{tw:niep} and \ref{thm:closedslice} allow to reduce the area needed to verify if a fractal has GLP. We also show that further reduction is impossible. We prove that a nested fractal can have a complex located in its center only in case of 3, 4 or 6 essential fixed points. If complexes are triangular or square, we always have GLP regardless of the structure of complexes, but in case of hexagonal complexes having a central complex ensures that the fractal does not have GLP.

Next, we present Theorem \ref{thm:2^n} which generalizes Corollary 3.1 from \cite{bib:KOPP} stating that fractals with 4 essential fixed points must have GLP. In Theorem \ref{thm:2^n} we prove that if a number of essential fixed points is a power of 2, then the fractal has GLP. Finally we show that primes and powers of 2 are the only numbers with that property. For any other number of essential fixed points we can construct examples having and other not having GLP.

%\textsc{Acknowledgements.} The authors wish to thank ...

\section{Preliminaries}

\subsection{Simple nested fractals} \label{sec:snf}

We follow the preliminary section of \cite{bib:KOPP}.

\begin{definition}
Let $\Psi_i : \mathbb{R}^2 \to \mathbb{R}^2$ for $i \in \{1,...,N\}$ be a collection of similitudes given by a formula
$$
\Psi_i(x) =  \frac{x}{L} + \nu_i,
$$
where $L>1$ is a scaling factor and $\nu_i \in \mathbb{R}^2$. %We shall assume $\nu_1 = 0$.
There exists a unique nonempty compact set $\mathcal{K}^{\left\langle 0\right\rangle}$ (called {\em the  fractal generated by the system} $(\Psi_i)_{i=1}^N$) such that $\mathcal{K}^{\left\langle 0\right\rangle} = \bigcup_{i=1}^{N} \Psi_i\left(\mathcal{K}^{\left\langle 0\right\rangle}\right)$. 
\end{definition}

\begin{definition}
Let $F$ be the collection of fixed points of the transformations $\Psi_1, ..., \Psi_N$. A fixed point $x \in F$ is an essential fixed point if there exists another fixed point $y \in F$ and two different similitudes $\Psi_i$, $\Psi_j$ such that $\Psi_i(x)=\Psi_j(y)$.
The set of all essential fixed points for transformations $\Psi_1, ..., \Psi_N$ is denoted by $V_{0}^{\left\langle 0\right\rangle}$. Moreover, let $k=\# V^{\left\langle 0\right\rangle}_{0}$.
\end{definition}

For example let us look at the Vicsek Cross. Its vertices (denoted as $ v_i $) are vertices of square with side $ a $. Here the similitudes $ \Psi_{i} $ are given by following formulas:
\begin{eqnarray*}
	\Psi_{1}(x)&=&\frac{1}{3} x+\left( \frac{2}{3}a,\frac{2}{3}a\right),\\
	\Psi_{2}(x)&=&\frac{1}{3} x+\left( 0,\frac{2}{3}a\right),\\
	\Psi_{3}(x)&=&\frac{1}{3} x,\\
	\Psi_{4}(x)&=&\frac{1}{3} x+\left( \frac{2}{3}a,0\right),\\
	\Psi_{5}(x)&=&\frac{1}{3} x+\left( \frac{1}{3}a,\frac{1}{3}a\right).
\end{eqnarray*}
The point $ v_1 $ is an essential fixed point, because $ \Psi_{3}(v_1)=\Psi_{5}(v_3) $. In the Figure \ref{Vicsek} their image is denoted by $ w $.
\begin{figure}[H]
\centering
\begin{tikzpicture}
	\coordinate (A) at (0,0) ;
	%tworzymy kwadrat
	\foreach \x in {1,2,3,4}{
		\path (A) ++(\x*90-45:1cm) coordinate (\x);
	}
	\filldraw[fill=lightgray,draw=black] (1) -- (2) -- (3) -- (4) -- cycle;
	
	\path (A) ++(90-45:2cm) coordinate (a);
	\foreach \x in {1,2,3,4}{
		\path (a) ++(\x*90-45:1cm) coordinate (a\x);
	}
	\filldraw[fill=lightgray,draw=black] (a1) -- (a2) -- (a3) -- (a4) -- cycle;
	
	\path (A) ++(90+45:2cm) coordinate (b);
	\foreach \x in {1,2,3,4}{
		\path (b) ++(\x*90-45:1cm) coordinate (b\x);
	}
	\filldraw[fill=lightgray,draw=black] (b1) -- (b2) -- (b3) -- (b4) -- cycle;
	
	\path (A) ++(-90-45:2cm) coordinate (c);
	\foreach \x in {1,2,3,4}{
		\path (c) ++(\x*90-45:1cm) coordinate (c\x);
	}
	\filldraw[fill=lightgray,draw=black] (c1) -- (c2) -- (c3) -- (c4) -- cycle;
	
	\path (A) ++(-90+45:2cm) coordinate (d);
	\foreach \x in {1,2,3,4}{
		\path (d) ++(\x*90-45:1cm) coordinate (d\x);
	}
	\filldraw[fill=lightgray,draw=black] (d1) -- (d2) -- (d3) -- (d4) -- cycle;
	
	\foreach \y in {a,b,c,d}{
		\foreach \x in {1,2,3,4}{
			\filldraw [black] (\y\x) circle (1.5pt);
		}
	}
	\node[above right] at (a1){$ v_1 $};
	\node[above left] at (b2){$ v_2 $};
	\node[below left] at (c3){$ v_3 $};
	\node[above right] at (c1){$ w $};
	\node[below right] at (d4){$ v_4 $};
\end{tikzpicture}
\caption{Vicsek Cross.}\label{Vicsek}
\end{figure}
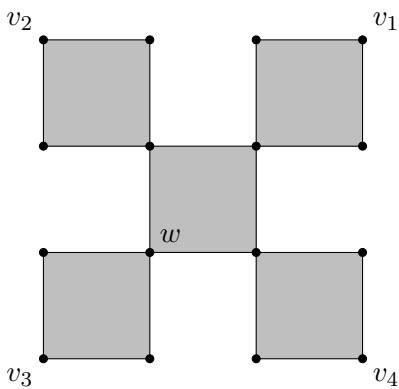

\begin{definition}[\textbf{Simple nested fractal}]
\label{def:snf}
 The fractal $\mathcal{K}^{\left\langle 0 \right\rangle}$ generated by the system $(\Psi_i)_{i=1}^N$ is called a \emph{simple nested fractal (SNF)} if the following conditions are met.
\begin{enumerate}
\item $\# V_{0}^{\left\langle 0\right\rangle} \geq 2.$
\item (Open Set Condition) There exists an open set $U \subset \mathbb{R}^2$ such that for $i\neq j$ one has\linebreak $\Psi_i (U) \cap \Psi_j (U)= \emptyset$ and $\bigcup_{i=1}^N \Psi_i (U) \subseteq U$.
\item (Nesting) $\Psi_i\left(\mathcal{K}^{\left\langle 0 \right\rangle}\right) \cap \Psi_j \left(\mathcal{K}^{\left\langle 0 \right\rangle}\right) = \Psi_i \left(V_{0}^{\left\langle 0\right\rangle}\right) \cap \Psi_j \left(V_{0}^{\left\langle 0\right\rangle}\right)$ for $i \neq j$.
\item (Symmetry) For $x,y \in V_{0}^{\left\langle 0\right\rangle},$ let $S_{x,y}$ denote the symmetry with respect to the line bisecting the segment $\left[x,y\right]$. Then
\begin{equation*}
\forall i \in \{1,...,M\} \ \forall x,y \in V_{0}^{\left\langle 0\right\rangle} \ \exists j \in \{1,...,M\} \ S_{x,y} \left( \Psi_i \left(V_{0}^{\left\langle 0\right\rangle} \right) \right) = \Psi_j \left(V_{0}^{\left\langle 0\right\rangle} \right).
\end{equation*}
\item (Connectivity) On the set $V_{-1}^{\left\langle 0\right\rangle}:= \bigcup_i \Psi_i \left(V_{0}^{\left\langle 0\right\rangle}\right)$ we define graph structure $E_{-1}$ as follows:\\
$(x,y) \in E_{-1}$ if and only if $x, y \in \Psi_i\left(\mathcal{K}^{\left\langle 0 \right\rangle}\right)$ for some $i$.\\
Then the graph $(V_{-1}^{\left\langle 0\right\rangle},E_{-1} )$ is required to be connected.
\end{enumerate}
\end{definition}

If  $\mathcal{K}^{\left\langle 0 \right\rangle}$ is a simple nested fractal, then we let
\begin{align} \label{eq:Kn}
\mathcal{K}^{\left\langle M\right\rangle} = L^M \mathcal{K}^{\left\langle 0\right\rangle}, \quad M \in \mathbb{Z},
\end{align}
and
\begin{align} \label{eq:Kinfty}
\mathcal{K}^{\left\langle \infty \right\rangle} = \bigcup_{M=0}^{\infty} \mathcal{K}^{\left\langle M\right\rangle}.
\end{align}
The set $\mathcal{K}^{\left\langle \infty \right\rangle}$ is the \textbf{unbounded simple nested fractal (USNF)}.

The remaining notions are collected in a single definition.
\begin{definition} Let $M\in\mathbb Z.$
\begin{itemize}
\item[(1)] $M$-complex: \label{def:Mcomplex}
every set $\Delta_M \subset \mathcal{K}^{\left\langle \infty \right\rangle}$ of the form
\begin{equation} \label{eq:Mcompl}
\Delta_M  = \mathcal{K}^{\left\langle M \right\rangle} + \nu_{\Delta_M},
\end{equation}
where $\nu_{\Delta_M}=\sum_{j=M+1}^{J} L^{j} \nu_{i_j},$ for some $J \geq M+1$, $\nu_{i_j} \in \left\{\nu_1,...,\nu_N\right\}$, is called an \emph{$M$-complex}.
%the sequence $I\left(\Delta_M\right) = \left(i_{M+1}, ..., i_{J}\right).$
\item[(2)] Vertices of the $M$-complex \eqref{eq:Mcompl}: the set $V\left(\Delta_M\right) =L^MV_0^{\langle 0 \rangle}+\nu_{\Delta_M}= L^{M} V^{\left\langle 0 \right\rangle}_0 + \sum_{j=M+1}^{J} L^{j} \nu_{i_j}$.
\item[(3)] Vertices of $\mathcal{K}^{\left\langle M \right\rangle}$:
$$
V^{\left\langle M\right\rangle}_{M} = V\left(\mathcal{K}^{\left\langle M \right\rangle}\right) = L^M V^{\left\langle 0\right\rangle}_{0}.
$$
\item[(4)] Vertices of all $M$-complexes inside a $(M+m)$-complex for $m>0$:
$$
V_M^{\langle M+m\rangle}= \bigcup_{i=1}^{N} V_M^{\langle M+m-1\rangle} + L^M \nu_i.
$$
\item[(5)] Vertices of all 0-complexes inside the unbounded nested fractal:
$$
V^{\left\langle \infty \right\rangle}_{0} = \bigcup_{M=0}^{\infty} V^{\left\langle M\right\rangle}_{0}.
$$
\item[(6)] Vertices of $M$-complexes from the unbounded fractal:
$$
V^{\left\langle \infty \right\rangle}_{M} = L^{M} V^{\left\langle \infty \right\rangle}_{0}
$$
\item[(7)] $\cH(\Delta_M) = \conv(V(\Delta_M))$ - the convex hull of the set of vertices of $\Delta_M$. It is a regular polygon with vertices being the vertices of the complex (see Proposition \ref{wielokat_foremny}).
\item[(8)] $ \mathcal{B}_{\mathcal{K}^{\langle M \rangle}} $ - the barycenter of $\mathcal{K}^{\langle M \rangle}$. Analogously we denote the barycenter of $\Delta_M$ as $ \mathcal{B}_{\Delta_M}$.
\end{itemize}
\end{definition}

In the Figure \ref{t.S.} we show $\mathcal{K}^{\langle M \rangle}$ for $ M\in\{0,1,2\} $ in case of the Sierpi\'nski gasket. We also show exemplary complexes $ \Delta_{0} $ and $ \Delta_{1} $.
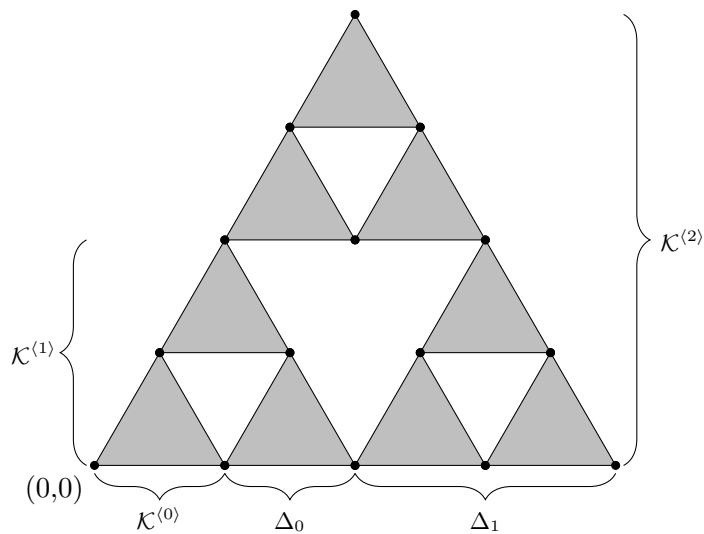
\begin{figure}[ht]
\centering
\begin{tikzpicture}
			\coordinate (A) at (0,0) ;
	%tworzymy pierwszy poziom trójkąta S.
	%1
	\foreach \x in {1,2,3}{
		\path (A) ++(\x*120+90:1cm) coordinate (\x);
	}
	\filldraw[fill=lightgray,draw=black] (1) -- (2) -- (3) -- cycle;
	\foreach \x in {1,2,3}{
		\filldraw [black] (\x) circle (1.5pt);
	}
	\node[below left] at (1) {(0,0)};
	%2
	\path (A) ++(-30:1cm)++(30:1cm) coordinate (B);
	\foreach \x in {1,2,3}{
		\path (B) ++(\x*120+90:1cm) coordinate (B\x);
	}
	\filldraw[fill=lightgray,draw=black] (B1) -- (B2) -- (B3) -- cycle;
	\foreach \x in {1,2,3}{
		\filldraw [black] (B\x) circle (1.5pt);
	}
	%3
	\path (B) ++(-30:1cm)++(30:1cm) coordinate (C);
	\foreach \x in {1,2,3}{
		\path (C) ++(\x*120+90:1cm) coordinate (C\x);
	}
	\filldraw[fill=lightgray,draw=black] (C1) -- (C2) -- (C3) -- cycle;
	\foreach \x in {1,2,3}{
		\filldraw [black] (C\x) circle (1.5pt);
	}
	%4
	\path (C) ++(-30:1cm)++(30:1cm) coordinate (D);
	\foreach \x in {1,2,3}{
		\path (D) ++(\x*120+90:1cm) coordinate (D\x);
	}
	\filldraw[fill=lightgray,draw=black] (D1) -- (D2) -- (D3) -- cycle;
	\foreach \x in {1,2,3}{
		\filldraw [black] (D\x) circle (1.5pt);
	}

	%drugi poziom
	%1
	\path (A) ++(90:1cm)++(30:1cm) coordinate (E);
	\foreach \x in {1,2,3}{
		\path (E) ++(\x*120+90:1cm) coordinate (E\x);
	}
	\filldraw[fill=lightgray,draw=black] (E1) -- (E2) -- (E3) -- cycle;
	\foreach \x in {1,2,3}{
		\filldraw [black] (E\x) circle (1.5pt);
	}
	%2
	\path (E) ++(-30:2cm)++(30:2cm) coordinate (F);
	\foreach \x in {1,2,3}{
		\path (F) ++(\x*120+90:1cm) coordinate (F\x);
	}
	\filldraw[fill=lightgray,draw=black] (F1) -- (F2) -- (F3) -- cycle;
	\foreach \x in {1,2,3}{
		\filldraw [black] (F\x) circle (1.5pt);
	}

	%trzeci poziom
	%1
	\path (E) ++(90:1cm)++(30:1cm) coordinate (G);
	\foreach \x in {1,2,3}{
		\path (G) ++(\x*120+90:1cm) coordinate (G\x);
	}
	\filldraw[fill=lightgray,draw=black] (G1) -- (G2) -- (G3) -- cycle;
	\foreach \x in {1,2,3}{
		\filldraw [black] (G\x) circle (1.5pt);
	}
	%2
	\path (G) ++(-30:1cm)++(30:1cm) coordinate (H);
	\foreach \x in {1,2,3}{
		\path (H) ++(\x*120+90:1cm) coordinate (H\x);
	}
	\filldraw[fill=lightgray,draw=black] (H1) -- (H2) -- (H3) -- cycle;
	\foreach \x in {1,2,3}{
		\filldraw [black] (H\x) circle (1.5pt);
	}

	%czwarty poziom
	%1
	\path (H) ++(90:1cm)++(150:1cm) coordinate (I);
	\foreach \x in {1,2,3}{
		\path (I) ++(\x*120+90:1cm) coordinate (I\x);
	}
	\filldraw[fill=lightgray,draw=black] (I1) -- (I2) -- (I3) -- cycle;
	\foreach \x in {1,2,3}{
		\filldraw [black] (I\x) circle (1.5pt);
	}

	%oznaczenia
	%k1
	\path (1) ++(0,3) coordinate(1');
	\draw [decorate,decoration={brace,amplitude=10pt,raise=3pt},xshift=-4pt,yshift=0pt]
	(1) -- (1') node [black,midway,xshift=-0.8cm] 
	{\footnotesize $\mathcal{K}^{\langle 1 \rangle}$};
	
	%k0
	\draw [decorate,decoration={brace,amplitude=10pt,mirror,raise=3pt},yshift=0pt]
	(1) -- (2) node [black,midway,yshift=-0.7cm] {\footnotesize
		$\mathcal{K}^{\langle 0 \rangle}$};
	
	%d0
	\draw [decorate,decoration={brace,amplitude=10pt,mirror,raise=3pt},yshift=0pt]
	(2) -- (B2) node [black,midway,yshift=-0.75cm] {\footnotesize
		$\Delta_{0}$};
	
	%d1
	\draw [decorate,decoration={brace,amplitude=10pt,mirror,raise=3pt},yshift=0pt]
	(C1) -- (D2) node [black,midway,yshift=-0.75cm] {\footnotesize
		$\Delta_{1}$};
	
	%k2
	\path (D2) ++(0,6) coordinate(D2');
	\draw [decorate,decoration={brace,amplitude=10pt,mirror,raise=3pt},yshift=0pt]
	(D2) -- (D2') node [black,midway,xshift=0.9cm] {\footnotesize
		$\mathcal{K}^{\langle 2 \rangle}$};
\end{tikzpicture}
\caption{The Sierpi\'nski Gasket.}\label{t.S.}
\end{figure}

We now recall the proposition, which is crucial for the next section.
	\begin{proposition}{(\cite{bib:KOPP}, Proposition 2.1.)}\label{wielokat_foremny}
		If $k \geq 3$, then points from $V_{0}^{\left\langle 0\right\rangle}$ are the vertices of a regular polygon.
		
If $k=2$, then $\mathcal{K}^{\left\langle 0 \right\rangle}$ is just a segment connecting $x_1$ and $x_2$.
	\end{proposition}
From now on we assume that $k \geq 3.$

\subsection{Good labelling property}\label{sec:glp}
In this section we will present basic definitions and theorems related to the good labeling property (GLP) of the USNF. The results are cited from \cite{bib:KOPP}.

  Consider the alphabet of $k$ symbols $\cA:=\left\{a_1, a_2,a_3,...,a_k\right\}$,  where  $ k=\# V^{\left\langle 0\right\rangle}_{0}\geq 3.$ The elements of $\cA$ will be called labels.
 \begin{definition}\label{def:labeling}
  Let $M \in \mathbb{Z}$. A \emph{labelling function of order $M$} is any map $l_M: V^{\left\langle \infty \right\rangle}_{M} \to \cA.$
  \end{definition}

Since the vertices of every $M$-complex $\Delta_M$ are the vertices of a regular polygon with $k$ vertices, there exist exactly $k$ different rotations around the barycenter of $\mathcal K^{\langle M \rangle}$,  mapping $V^{\left\langle M \right\rangle}_{M}$ onto $V^{\left\langle M \right\rangle}_{M}$. They will be denoted by  $\{R_1, ..., R_k\}=: \mathcal{R}_M $ (the rotations are ordered in such a way that for $i=1,2,...,k,$ the rotation $R_i$ rotates by angle $\frac{2\pi i}{k})$.

\begin{definition}{\cite[Definition 3.2]{bib:KOPP}}[\textbf{Good labelling function of order $M$}] 
\label{def:glp} Let $M \in \mathbb{Z}$.  A function $\ell_M: V^{\left\langle \infty \right\rangle}_{M} \to \cA$  is called a \emph{good labelling function of order $M$} if the following conditions are met.
\begin{itemize}
\item[(1)] The restriction of $\ell_M$ to $V^{\left\langle M \right\rangle}_{M}$ is a bijection onto $\cA$.
\item[(2)] For every $M$-complex $\Delta_M$ represented as
$$
\Delta_M  = \mathcal{K}^{\left\langle M \right\rangle} +\nu_{\Delta_M},
$$
where $\nu_{\Delta_M}=  \sum_{j=M+1}^{J} L^{j} \nu_{i_j},$  with some $J \geq M+1$ and $\nu_{i_j} \in \left\{\nu_1,...,\nu_N\right\}$ (cf. Def. \ref{def:Mcomplex}), there exists a rotation $R_{\Delta_M} \in \mathcal{R}_M$ such that
\begin{align}\label{eq:rot}
\ell_M(v)=\ell_M\left(R_{\Delta_M}\left(v -\nu_{\Delta_M}\right)\right) , \quad v \in V\left(\Delta_M\right).
\end{align}
\end{itemize}
An USNF $\mathcal{K}^{\left\langle \infty \right\rangle}$ is said to have the  \emph{good labelling property of order $M$} if a good labelling function of order $M$ exists.
\end{definition}

Note that for every $M$-complex $\Delta_M$  the restriction of a good labelling function to $V\left(\Delta_M\right)$ is a bijection onto $\cA$.

Thanks to the selfsimilar structure of $\mathcal{K}^{\left\langle \infty \right\rangle},$ the \emph{good labelling property of order $M$} for some $M \in \mathbb{Z}$ is equivalent to this property of any other order $\widetilde{M} \in \mathbb{Z}$. This gives rise to the following general definition.

\begin{definition}{\cite[Definition 3.3]{bib:KOPP}}[\textbf{Good labelling property}] \label{def:glp_gen} 
An USNF $\mathcal{K}^{\left\langle \infty \right\rangle}$ is said to have the \emph{good labelling property (GLP in short)} if it has the {good labelling property of order $M$} for some $M \in \mathbb{Z}$.
\end{definition}

\begin{proposition}{\cite[Proposition 3.1]{bib:KOPP}}
For USNF's with the GLP, for any $M\in \mathbb Z$ the good labelling of order $M$ is unique up to a permutation of the alphabet set $\cA.$ In particular,
if $\mathcal{K}^{\left\langle \infty \right\rangle}$ has  the GLP and
a bijection $\widetilde{\ell}_M: V^{\left\langle M \right\rangle}_{M} \to \mathcal{A}$ is given, then there exists a unique
good labelling function $\ell_M:V_M^{\langle \infty\rangle} \to \cA$ such that $\ell_M|_{V_{M}^{\langle M\rangle}}=\widetilde \ell_M.$
\end{proposition}

\begin{proposition}{\cite[Proposition 3.2]{bib:KOPP}}
Let $\mathcal{K}^{\left\langle \infty \right\rangle}$ be a planar USNF, $M\in \mathbb Z$ and let ${\ell}_{M,0}: V^{\left\langle M \right\rangle}_{M} \to \mathcal{A}$ be a bijection. Then $\mathcal{K}^{\left\langle \infty \right\rangle}$ has the GLP if and only if there exists an extension of ${\ell}_{M,0}$ to $\widetilde{\ell}_{M,0}: V^{\left\langle M+1 \right\rangle}_{M} \to \mathcal{A}$ such that for every $M$-complex $\Delta_M \subset \mathcal{K}^{\left\langle M+1 \right\rangle}$ represented as (cf. Definition \ref{def:Mcomplex})
\begin{equation*}
\Delta_M = \mathcal{K}^{\left\langle M \right\rangle} + L^{M+1} \nu_{i_{M+1}}, \quad \nu_{i_{M+1}} \in \left\{\nu_1,...,\nu_N\right\},
\end{equation*}
there exists a rotation $R_{\Delta_M} \in \mathcal{R}_M$ such that
\begin{equation} \label{eq:aux_agr}
\widetilde{\ell}_{M,0}\left(R_{\Delta_M}\left(v - L^{M+1} \nu_{i_{M+1}}\right)\right) = \widetilde{\ell}_{M,0}(v), \quad v \in V\left(\Delta_M\right).
\end{equation}
\end{proposition}

This proposition means that if a labeling on $V_M^{\langle M\rangle}$ can be extended in a `good' way to $V_M^{\langle M+1\rangle}$,
then it can be extended as a good labelling also to $V_M^{\langle\infty\rangle}$.

Equivalently, due to self-similarity, it is sufficient to show that the labeling on $V_0^{\langle 0\rangle}$ can be extended to $V_0^{\langle 1\rangle}$.
In order to simplify the proofs in the next section, we will describe further results in terms of labeling of $V_0^{\langle 1\rangle}$.

Checking manually if the fractal has GLP may not be easy. In \cite{bib:KOPP} the following theorems were proved.

\begin{theorem} \cite[Theorems 3.1, 3.2]{bib:KOPP}
\label{thm:prime}
If $\# V_{0}^{\left\langle 0\right\rangle} = p$, where $p\geq 3$ is a prime number, then $\mathcal{K}^{\left\langle \infty \right\rangle}$ has the GLP.
\end{theorem}

	\begin{figure}[ht]
	\begin{center}
		\begin{tikzpicture}
			%5
			\coordinate (A) at (0,0);
			%tworzymy pięciokąt foremny
			\path (A) ++(18:1cm) coordinate (0) node[above]{A};
			\path (A) ++(1*72+18:1cm) coordinate (1) node[below]{B};
			\path (A) ++(2*72+18:1cm) coordinate (2) node[left]{C};
			\path (A) ++(3*72+18:1cm) coordinate (3) node[below]{D};
			\path (A) ++(4*72+18:1cm) coordinate (4) node[below]{E};
			\filldraw[fill=lightgray]  (0) -- (1) -- (2) -- (3) -- (4) -- cycle;
			\path (A) ++(1*72+18:1cm) coordinate (1) node[below, yshift=-0.1cm]{B};
			%drugi
			\path (A) ++(-18:1cm)++(18:1cm) coordinate (B);
			\path (B) ++(1*72+18:1cm) coordinate (B1) node[above]{E};
			\path (B) ++(2*72+18:1cm) coordinate (B2);
			\path (B) ++(3*72+18:1cm) coordinate (B3) node[below]{B};
			\path (B) ++(4*72+18:1cm) coordinate (B4) node[below]{C};
			\path (B) ++(5*72+18:1cm) coordinate (B5) node[above]{D};
			\filldraw[fill=lightgray] (B1) -- (B2) -- (B3) -- (B4) -- (B5) -- cycle;
			%trzeci
			\path (B) ++(-18:1cm)++(18:1cm) coordinate (C);
			\path (C) ++(1*72+18:1cm) coordinate (C1) node[below]{C};
			\path (C) ++(2*72+18:1cm) coordinate (C2);
			\path (C) ++(3*72+18:1cm) coordinate (C3) node[below]{E};
			\path (C) ++(4*72+18:1cm) coordinate (C4) node[below]{A};
			\path (C) ++(5*72+18:1cm) coordinate (C5) node[right]{B};
			\filldraw[fill=lightgray] (C1) -- (C2) -- (C3) -- (C4) -- (C5) -- cycle;
			\path (C) ++(1*72+18:1cm) coordinate (C1) node[below, yshift=-0.1cm]{C};
			%czwarty
			\path (C) ++(90:1cm)++(54:1cm) coordinate (D);
			\path (D) ++(1*72+18:1cm) coordinate (D1);
			\path (D) ++(2*72+18:1cm) coordinate (D2) node[left]{B};
			\path (D) ++(3*72+18:1cm) coordinate (D3);
			\path (D) ++(4*72+18:1cm) coordinate (D4) node[right]{D};
			\path (D) ++(5*72+18:1cm) coordinate (D5) node[right]{E};
			\filldraw[fill=lightgray] (D1) -- (D2) -- (D3) -- (D4) -- (D5) -- cycle;
			\path (D) ++(1*72+18:1cm) coordinate (D1) node[below, yshift=-0.1cm]{A};
			%piąty
			\path (D) ++(90:1cm)++(54:1cm) coordinate (E);
			\path (E) ++(1*72+18:1cm) coordinate (E1) node[above]{D};
			\path (E) ++(2*72+18:1cm) coordinate (E2);
			\path (E) ++(3*72+18:1cm) coordinate (E3);
			\path (E) ++(4*72+18:1cm) coordinate (E4) node[right]{B};
			\path (E) ++(5*72+18:1cm) coordinate (E5) node[right]{C};
			\filldraw[fill=lightgray] (E1) -- (E2) -- (E3) -- (E4) -- (E5) -- cycle;
			\path (E) ++(2*72+18:1cm) coordinate (E2) node[below, xshift=-0.2cm]{E};
			%szósty
			\path (E) ++(162:1cm)++(126:1cm) coordinate (F);
			\path (F) ++(1*72+18:1cm) coordinate (F1) node[above]{B};
			\path (F) ++(2*72+18:1cm) coordinate (F2);
			\path (F) ++(3*72+18:1cm) coordinate (F3) node[below]{D};
			\path (F) ++(4*72+18:1cm) coordinate (F4);
			\path (F) ++(5*72+18:1cm) coordinate (F5) node[above]{A};
			\filldraw[fill=lightgray] (F1) -- (F2) -- (F3) -- (F4) -- (F5) -- cycle;
			\path (F) ++(2*72+18:1cm) coordinate (F2) node[below, xshift=-0.2cm]{C};
			%siódmy
			\path (F) ++(162:1cm)++(126:1cm) coordinate (G);
			\path (G) ++(1*72+18:1cm) coordinate (G1) node[above]{E};
			\path (G) ++(2*72+18:1cm) coordinate (G2) node[above]{A};
			\path (G) ++(3*72+18:1cm) coordinate (G3);
			\path (G) ++(4*72+18:1cm) coordinate (G4);
			\path (G) ++(5*72+18:1cm) coordinate (G5) node[above]{D};
			\filldraw[fill=lightgray] (G1) -- (G2) -- (G3) -- (G4) -- (G5) -- cycle;
			%ósmy
			\path (G) ++(18+3*72:1cm)++(198:1cm) coordinate (H);
			\path (H) ++(1*72+18:1cm) coordinate (H1) node[above]{C};
			\path (H) ++(2*72+18:1cm) coordinate (H2) node[above]{D};
			\path (H) ++(3*72+18:1cm) coordinate (H3);
			\path (H) ++(4*72+18:1cm) coordinate (H4)node[below]{A};
			\path (H) ++(5*72+18:1cm) coordinate (H5);
			\filldraw[fill=lightgray] (H1) -- (H2) -- (H3) -- (H4) -- (H5) -- cycle;
			\path (G) ++(3*72+18:1cm) coordinate (G3) node[below, xshift=0.2cm]{B};						
			%dziewiąty
			\path (H) ++(18+3*72:1cm)++(198:1cm) coordinate (I);
			\path (I) ++(1*72+18:1cm) coordinate (I1) node[above]{A};
			\path (I) ++(2*72+18:1cm) coordinate (I2) node[left]{B};
			\path (I) ++(3*72+18:1cm) coordinate (I3) node[left]{C};
			\path (I) ++(4*72+18:1cm) coordinate (I4);
			\path (I) ++(5*72+18:1cm) coordinate (I5);
			\filldraw[fill=lightgray] (I1) -- (I2) -- (I3) -- (I4) -- (I5) -- cycle;
			\path (H) ++(3*72+18:1cm) coordinate (H3) node[below, xshift=0.2cm]{E};
			%dziesiąty
			\path (I) ++(18-72:1cm)++(-90:1cm) coordinate (J);
			\path (J) ++(1*72+18:1cm) coordinate (J1);
			\path (J) ++(2*72+18:1cm) coordinate (J2) node[left]{E};
			\path (J) ++(3*72+18:1cm) coordinate (J3) node[left]{A};
			\path (J) ++(4*72+18:1cm) coordinate (J4);
			\path (J) ++(5*72+18:1cm) coordinate (J5) node[right]{C};
			\filldraw[fill=lightgray] (J1) -- (J2) -- (J3) -- (J4) -- (J5) -- cycle;
			\path (I) ++(4*72+18:1cm) coordinate (I4)node[below, yshift=-0.1cm]{D};
			%wierzchołki
			\filldraw [black] (0) circle (1.5pt);
			\filldraw [black] (1) circle (1.5pt);
			\filldraw [black] (2) circle (1.5pt);
			\filldraw [black] (3) circle (1.5pt);
			\filldraw [black] (4) circle (1.5pt);
			\filldraw [black] (B1) circle (1.5pt);
			\filldraw [black] (B2) circle (1.5pt);
			\filldraw [black] (B3) circle (1.5pt);
			\filldraw [black] (B4) circle (1.5pt);
			\filldraw [black] (B5) circle (1.5pt);
			\filldraw [black] (C1) circle (1.5pt);
			\filldraw [black] (C2) circle (1.5pt);
			\filldraw [black] (C3) circle (1.5pt);
			\filldraw [black] (C4) circle (1.5pt);
			\filldraw [black] (C5) circle (1.5pt);
			\filldraw [black] (D1) circle (1.5pt);
			\filldraw [black] (D2) circle (1.5pt);
			\filldraw [black] (D3) circle (1.5pt);
			\filldraw [black] (D4) circle (1.5pt);
			\filldraw [black] (D5) circle (1.5pt);
			\filldraw [black] (E1) circle (1.5pt);
			\filldraw [black] (E2) circle (1.5pt);
			\filldraw [black] (E3) circle (1.5pt);
			\filldraw [black] (E4) circle (1.5pt);
			\filldraw [black] (E5) circle (1.5pt);
			\filldraw [black] (F1) circle (1.5pt);
			\filldraw [black] (F2) circle (1.5pt);
			\filldraw [black] (F3) circle (1.5pt);
			\filldraw [black] (F4) circle (1.5pt);
			\filldraw [black] (F5) circle (1.5pt);
			\filldraw [black] (G1) circle (1.5pt);
			\filldraw [black] (G2) circle (1.5pt);
			\filldraw [black] (G3) circle (1.5pt);
			\filldraw [black] (G4) circle (1.5pt);
			\filldraw [black] (G5) circle (1.5pt);
			\filldraw [black] (H1) circle (1.5pt);
			\filldraw [black] (H2) circle (1.5pt);
			\filldraw [black] (H3) circle (1.5pt);
			\filldraw [black] (H4) circle (1.5pt);
			\filldraw [black] (H5) circle (1.5pt);
			\filldraw [black] (I1) circle (1.5pt);
			\filldraw [black] (I2) circle (1.5pt);
			\filldraw [black] (I3) circle (1.5pt);
			\filldraw [black] (I4) circle (1.5pt);
			\filldraw [black] (I5) circle (1.5pt);
			\filldraw [black] (J1) circle (1.5pt);
			\filldraw [black] (J2) circle (1.5pt);
			\filldraw [black] (J3) circle (1.5pt);
			\filldraw [black] (J4) circle (1.5pt);
			\filldraw [black] (J5) circle (1.5pt);
		\end{tikzpicture}
	\end{center}
	\caption{Well labeled $\cK^{\langle 1 \rangle}$ in case of pentagonal complexes.}
	\end{figure}
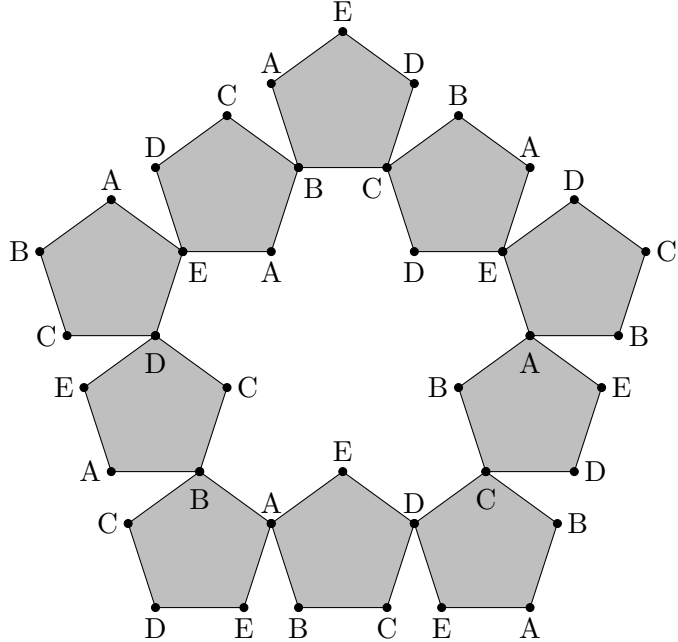

\begin{theorem} \cite[Theorem 3.4]{bib:KOPP}
\label{thm:even}
If $\# V_{0}^{\left\langle 0\right\rangle}=k$, $k \geq 3$, is an even number, then $\mathcal{K}^{\left\langle \infty \right\rangle}$ has the  GLP if and only if the $0$-complexes inside the $1$-complex $\mathcal{K}^{\left\langle 1 \right\rangle}$ can be split into two disjoint classes such that each complex from one of the classes intersects only complexes from the other class.
\end{theorem}

	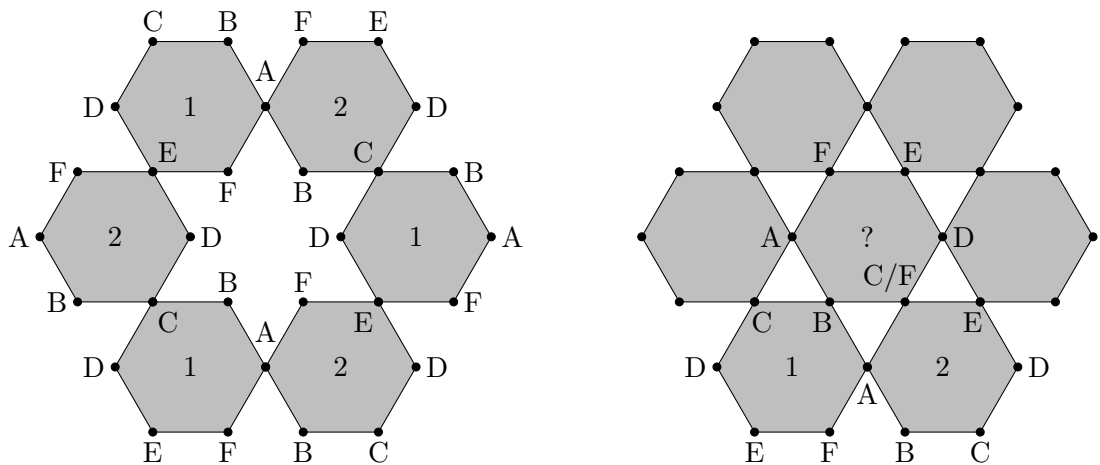
\begin{figure}[ht]
	\centering
	\begin{tikzpicture}
		%6
		\coordinate (A) at (0,0) node{1};
		%tworzymy sześciokąt foremny
		\path (A) ++(0:1cm) coordinate (0) node[above,yshift=0.2cm]{A};
		\path (A) ++(1*60:1cm) coordinate (1) node[above]{B};
		\path (A) ++(2*60:1cm) coordinate (2) node[below,xshift=0.2cm]{C};
		\path (A) ++(3*60:1cm) coordinate (3) node[left]{D};
		\path (A) ++(4*60:1cm) coordinate (4) node[below]{E};
		\path (A) ++(5*60:1cm) coordinate (5) node[below]{F};
		\filldraw[fill=lightgray,draw=black] (0) -- (1) -- (2) -- (3) -- (4) -- (5) -- cycle;
		\coordinate (A) at (0,0) node{1};
		\path (A) ++(2*60:1cm) coordinate (2) node[below,xshift=0.2cm]{C};
		%drugi
		\path (A) ++(0:2cm) coordinate (B) node{2};
		\path (B) ++(0:1cm) coordinate (B0) node[right]{D};
		\path (B) ++(1*60:1cm) coordinate (B1) node[below,xshift=-0.2cm]{E};
		\path (B) ++(2*60:1cm) coordinate (B2) node[above]{F};
		\path (B) ++(3*60:1cm) coordinate (B3);
		\path (B) ++(4*60:1cm) coordinate (B4) node[below]{B};
		\path (B) ++(5*60:1cm) coordinate (B5) node[below]{C};
		\filldraw[fill=lightgray,draw=black] (B0) -- (B1) -- (B2) -- (B3) -- (B4) -- (B5) -- cycle;
		\path (B) ++(1*60:1cm) coordinate (B1) node[below,xshift=-0.2cm]{E};
		\path (A) ++(0:2cm) coordinate (B) node{2};
		%trzeci
		\path (B) ++(60:2cm) coordinate (C) node{1};
		\path (C) ++(0:1cm) coordinate (C0) node[right]{A};
		\path (C) ++(1*60:1cm) coordinate (C1) node[right]{B};
		\path (C) ++(2*60:1cm) coordinate (C2) node[above,xshift=-0.2cm]{C};
		\path (C) ++(3*60:1cm) coordinate (C3)  node[left]{D};
		\path (C) ++(4*60:1cm) coordinate (C4);
		\path (C) ++(5*60:1cm) coordinate (C5) node[right]{F};
		\filldraw[fill=lightgray,draw=black] (C0) -- (C1) -- (C2) -- (C3) -- (C4) -- (C5) -- cycle;
		\path (B) ++(60:2cm) coordinate (C) node{1};
		%czwarty
		\path (C) ++(120:2cm) coordinate (D) node{2};
		\path (D) ++(0:1cm) coordinate (D0) node[right]{D};
		\path (D) ++(1*60:1cm) coordinate (D1) node[above]{E};
		\path (D) ++(2*60:1cm) coordinate (D2) node[above]{F};
		\path (D) ++(3*60:1cm) coordinate (D3)  node[above,yshift=0.2cm]{A};
		\path (D) ++(4*60:1cm) coordinate (D4) node[below]{B};
		\path (D) ++(5*60:1cm) coordinate (D5);
		\filldraw[fill=lightgray,draw=black] (D0) -- (D1) -- (D2) -- (D3) -- (D4) -- (D5) -- cycle;
		\path (C) ++(2*60:1cm) coordinate (C2) node[above,xshift=-0.2cm]{C};
		\path (C) ++(120:2cm) coordinate (D) node{2};
		%piąty
		\path (D) ++(180:2cm) coordinate (E) node{1};
		\path (E) ++(0:1cm) coordinate (E0);
		\path (E) ++(1*60:1cm) coordinate (E1) node[above]{B};
		\path (E) ++(2*60:1cm) coordinate (E2) node[above]{C};
		\path (E) ++(3*60:1cm) coordinate (E3)  node[left]{D};
		\path (E) ++(4*60:1cm) coordinate (E4) node[above,xshift=0.2cm]{E};
		\path (E) ++(5*60:1cm) coordinate (E5) node[below]{F};
		\filldraw[fill=lightgray,draw=black] (E0) -- (E1) -- (E2) -- (E3) -- (E4) -- (E5) -- cycle;
		\path (D) ++(180:2cm) coordinate (E) node{1};
		\path (E) ++(4*60:1cm) coordinate (E4) node[above,xshift=0.2cm]{E};
		%szósty
		\path (E) ++(240:2cm) coordinate (F) node{2};
		\path (F) ++(0:1cm) coordinate (F0) node[right]{D};
		\path (F) ++(1*60:1cm) coordinate (F1);
		\path (F) ++(2*60:1cm) coordinate (F2) node[left]{F};
		\path (F) ++(3*60:1cm) coordinate (F3)  node[left]{A};
		\path (F) ++(4*60:1cm) coordinate (F4) node[left]{B};
		\path (F) ++(5*60:1cm) coordinate (F5);
		\filldraw[fill=lightgray,draw=black] (F0) -- (F1) -- (F2) -- (F3) -- (F4) -- (F5) -- cycle;		
		\path (E) ++(240:2cm) coordinate (F) node{2};
		%wierzchołki
		\filldraw [black] (0) circle (1.5pt);
		\filldraw [black] (1) circle (1.5pt);
		\filldraw [black] (2) circle (1.5pt);
		\filldraw [black] (3) circle (1.5pt);
		\filldraw [black] (4) circle (1.5pt);
		\filldraw [black] (5) circle (1.5pt);
		\filldraw [black] (B1) circle (1.5pt);
		\filldraw [black] (B2) circle (1.5pt);
		\filldraw [black] (B3) circle (1.5pt);
		\filldraw [black] (B4) circle (1.5pt);
		\filldraw [black] (B0) circle (1.5pt);
		\filldraw [black] (B5) circle (1.5pt);
		\filldraw [black] (C0) circle (1.5pt);
		\filldraw [black] (C1) circle (1.5pt);
		\filldraw [black] (C2) circle (1.5pt);
		\filldraw [black] (C3) circle (1.5pt);
		\filldraw [black] (C4) circle (1.5pt);
		\filldraw [black] (C5) circle (1.5pt);
		\filldraw [black] (D0) circle (1.5pt);
		\filldraw [black] (D1) circle (1.5pt);
		\filldraw [black] (D2) circle (1.5pt);
		\filldraw [black] (D3) circle (1.5pt);
		\filldraw [black] (D4) circle (1.5pt);
		\filldraw [black] (D5) circle (1.5pt);
		\filldraw [black] (E0) circle (1.5pt);
		\filldraw [black] (E1) circle (1.5pt);
		\filldraw [black] (E2) circle (1.5pt);
		\filldraw [black] (E3) circle (1.5pt);
		\filldraw [black] (E4) circle (1.5pt);
		\filldraw [black] (E5) circle (1.5pt);
		\filldraw [black] (F0) circle (1.5pt);
		\filldraw [black] (F1) circle (1.5pt);
		\filldraw [black] (F2) circle (1.5pt);
		\filldraw [black] (F3) circle (1.5pt);
		\filldraw [black] (F4) circle (1.5pt);
		\filldraw [black] (F5) circle (1.5pt);	
		
		%6 nr. 2
		\coordinate (1A) at (8,0);
		\path (1A) ++(0:1cm) ++(180:1cm) node{1};
		%tworzymy sześciokąt foremny
		\path (1A) ++(0:1cm) coordinate (10) node[below, yshift=-0.1cm]{A};
		\path (1A) ++(1*60:1cm) coordinate (11);
		\path (1A) ++(2*60:1cm) coordinate (12) node[below, xshift=0.1cm]{C};
		\path (1A) ++(3*60:1cm) coordinate (13) node[left]{D};
		\path (1A) ++(4*60:1cm) coordinate (14) node[below]{E};
		\path (1A) ++(5*60:1cm) coordinate (15) node[below]{F};
		\filldraw[fill=lightgray,draw=black] (10) -- (11) -- (12) -- (13) -- (14) -- (15) -- cycle;
		\path (1A) ++(0:1cm) ++(180:1cm) node{1};
		\path (1A) ++(2*60:1cm) coordinate (12) node[below, xshift=0.1cm]{C};
		%drugi
		\path (1A) ++(0:2cm) coordinate (1B) node{2};
		\path (1B) ++(0:1cm) coordinate (1B0) node[right]{D};
		\path (1B) ++(1*60:1cm) coordinate (1B1) node[below, xshift=-0.1cm]{E};
		\path (1B) ++(2*60:1cm) coordinate (1B2) node[above, xshift=-0.2cm]{C/F};
		\path (1B) ++(3*60:1cm) coordinate (1B3);
		\path (1B) ++(4*60:1cm) coordinate (1B4) node[below]{B};
		\path (1B) ++(5*60:1cm) coordinate (1B5) node[below]{C};
		\filldraw[fill=lightgray,draw=black] (1B0) -- (1B1) -- (1B2) -- (1B3) -- (1B4) -- (1B5) -- cycle;
		\path (1B) ++(1*60:1cm) coordinate (1B1) node[below, xshift=-0.1cm]{E};
		\path (1A) ++(0:2cm) coordinate (1B) node{2};
		%trzeci
		\path (1B) ++(60:2cm) coordinate (1C);
		\path (1C) ++(0:1cm) coordinate (1C0);
		\path (1C) ++(1*60:1cm) coordinate (1C1) ;
		\path (1C) ++(2*60:1cm) coordinate (1C2) ;
		\path (1C) ++(3*60:1cm) coordinate (1C3)  ;
		\path (1C) ++(4*60:1cm) coordinate (1C4);
		\path (1C) ++(5*60:1cm) coordinate (1C5) ;
		\filldraw[fill=lightgray,draw=black] (1C0) -- (1C1) -- (1C2) -- (1C3) -- (1C4) -- (1C5) -- cycle;
		%czwarty
		\path (1C) ++(120:2cm) coordinate (1D) ;
		\path (1D) ++(0:1cm) coordinate (1D0) ;
		\path (1D) ++(1*60:1cm) coordinate (1D1) ;
		\path (1D) ++(2*60:1cm) coordinate (1D2) ;
		\path (1D) ++(3*60:1cm) coordinate (1D3)  ;
		\path (1D) ++(4*60:1cm) coordinate (1D4) ;
		\path (1D) ++(5*60:1cm) coordinate (1D5);
		\filldraw[fill=lightgray,draw=black] (1D0) -- (1D1) -- (1D2) -- (1D3) -- (1D4) -- (1D5) -- cycle;
		%piąty
		\path (1D) ++(180:2cm) coordinate (1E) ;
		\path (1E) ++(0:1cm) coordinate (1E0);
		\path (1E) ++(1*60:1cm) coordinate (1E1) ;
		\path (1E) ++(2*60:1cm) coordinate (1E2) ;
		\path (1E) ++(3*60:1cm) coordinate (1E3);
		\path (1E) ++(4*60:1cm) coordinate (1E4) ;
		\path (1E) ++(5*60:1cm) coordinate (1E5) ;
		\filldraw[fill=lightgray,draw=black] (1E0) -- (1E1) -- (1E2) -- (1E3) -- (1E4) -- (1E5) -- cycle;
		%szósty
		\path (1E) ++(240:2cm) coordinate (1F) ;
		\path (1F) ++(0:1cm) coordinate (1F0) ;
		\path (1F) ++(1*60:1cm) coordinate (1F1);
		\path (1F) ++(2*60:1cm) coordinate (1F2) ;
		\path (1F) ++(3*60:1cm) coordinate (1F3)  ;
		\path (1F) ++(4*60:1cm) coordinate (1F4) ;
		\path (1F) ++(5*60:1cm) coordinate (1F5);
		\filldraw[fill=lightgray,draw=black] (1F0) -- (1F1) -- (1F2) -- (1F3) -- (1F4) -- (1F5) -- cycle;
		%siódmy
		\path (1F) ++(0:2cm) coordinate (1G) node{?};
		\path (1G) ++(0:1cm) coordinate (1G0) node[right]{D};
		\path (1G) ++(1*60:1cm) coordinate (1G1) node[above, xshift=0.1cm]{E};
		\path (1G) ++(2*60:1cm) coordinate (1G2) node[above, xshift=-0.1cm]{F};
		\path (1G) ++(3*60:1cm) coordinate (1G3)  node[left]{A};
		\path (1G) ++(4*60:1cm) coordinate (1G4) node[below, xshift=-0.1cm]{B};
		\path (1G) ++(5*60:1cm) coordinate (1G5) ;
		\filldraw[fill=lightgray,draw=black] (1G0) -- (1G1) -- (1G2) -- (1G3) -- (1G4) -- (1G5) -- cycle;
		\path (1F) ++(0:2cm) coordinate (1G) node{?};	
		\path (1B) ++(2*60:1cm) coordinate (1B2) node[above, xshift=-0.2cm]{C/F};			
		%wierzchołki
		\filldraw [black] (10) circle (1.5pt);
		\filldraw [black] (11) circle (1.5pt);
		\filldraw [black] (12) circle (1.5pt);
		\filldraw [black] (13) circle (1.5pt);
		\filldraw [black] (14) circle (1.5pt);
		\filldraw [black] (15) circle (1.5pt);
		\filldraw [black] (1B1) circle (1.5pt);
		\filldraw [black] (1B2) circle (1.5pt);
		\filldraw [black] (1B3) circle (1.5pt);
		\filldraw [black] (1B4) circle (1.5pt);
		\filldraw [black] (1B0) circle (1.5pt);
		\filldraw [black] (1B5) circle (1.5pt);
		\filldraw [black] (1C0) circle (1.5pt);
		\filldraw [black] (1C1) circle (1.5pt);
		\filldraw [black] (1C2) circle (1.5pt);
		\filldraw [black] (1C3) circle (1.5pt);
		\filldraw [black] (1C4) circle (1.5pt);
		\filldraw [black] (1C5) circle (1.5pt);
		\filldraw [black] (1D0) circle (1.5pt);
		\filldraw [black] (1D1) circle (1.5pt);
		\filldraw [black] (1D2) circle (1.5pt);
		\filldraw [black] (1D3) circle (1.5pt);
		\filldraw [black] (1D4) circle (1.5pt);
		\filldraw [black] (1D5) circle (1.5pt);
		\filldraw [black] (1E0) circle (1.5pt);
		\filldraw [black] (1E1) circle (1.5pt);
		\filldraw [black] (1E2) circle (1.5pt);
		\filldraw [black] (1E3) circle (1.5pt);
		\filldraw [black] (1E4) circle (1.5pt);
		\filldraw [black] (1E5) circle (1.5pt);
		\filldraw [black] (1F0) circle (1.5pt);
		\filldraw [black] (1F1) circle (1.5pt);
		\filldraw [black] (1F2) circle (1.5pt);
		\filldraw [black] (1F3) circle (1.5pt);
		\filldraw [black] (1F4) circle (1.5pt);
		\filldraw [black] (1F5) circle (1.5pt);					
	\end{tikzpicture}
	\caption{Example of labeling and division into classes for the Sierpi\'nski Hexagon and an attempt of labeling of the Lindstr\o m Snowflake.}\label{fig:hexagon}
	\end{figure}

The Sierpi\'nski Hexagon has the GLP. Its vertices can be labeled $A$-$F$ as in the Figure \ref{fig:hexagon} and its complexes can be divided into classes $1$-$2$. The labeling of the Lindstr\o m Snowflake is impossible. In the Figure \ref{fig:hexagon} the middle complex cannot be in the first, nor in the second class. The vertex $C/F$ can have only one label, hence the good labeling function satisfying the Definition \ref{def:glp} does not exist.

\begin{theorem} \cite[Theorem 3.3]{bib:KOPP}
\label{thm:cycle}
If {$k \geq 3$} and there are no inessential fixed points of the similitudes generating $\mathcal{K}^{\langle 0\rangle},$ i.e. $k = N$,
then $\mathcal{K}^{\left\langle \infty \right\rangle}$ has the GLP.
\end{theorem}

Summing up, it was known that all fractals with $k$ prime or $k=N$ have GLP. Moreover we knew the characterization of GLP for fractals with $k$ even.

We still have to analyze fractals with $k$ odd, composite, and $N>k$. We will focus on them in the next chapter.

\section{New results}

\subsection{Characterization of GLP for fractals with odd $k$}
	 
	 Theorem \ref{thm:odd} is the counterpart of Theorem \ref{thm:even} for fractals with odd $k$. It gives characterization of GLP and conditions this property on the geometric structure of complexes. For odd $k$ we can get Theorem \ref{thm:cycle} as a special case of Theorem \ref{thm:odd}.
	 
	 In the next theorem as a cycle of 0-complexes we will understand a sequence of adjacent 0-complexes $\Delta_{0,1},\Delta_{0,2},\dots,\Delta_{0,j}$, for some $j$, (all distinct) such that the last $ \Delta_{0,j}$ is adjacent to the first $\Delta_{0,1}$.
	 
		We say that a rotation by angle $\alpha$ (counter-clockwise) occurs in a cycle if the next complex arises from the previous one by rotation by the angle $\alpha$ around their common vertex (Figure \ref{obrót}). From Fact 3.1 in \cite{bib:KOPP} we know that for odd $k$ only rotations by $\frac{k+1}{k} \pi$ and $\frac{k-1}{k} \pi$ may occur in the cycle.

	\begin{center}
		\begin{figure}[h]
		\centering
		\begin{tikzpicture}
			%5
			\coordinate (A) at (0,0);
			%tworzymy pięciokąt foremny
			\path (A) ++(18:1cm) coordinate (0) node[above]{A};
			\path (A) ++(1*72+18:1cm) coordinate (1) node[above]{B};
			\path (A) ++(2*72+18:1cm) coordinate (2) node[left]{C};
			\path (A) ++(3*72+18:1cm) coordinate (3) node[below]{D};
			\path (A) ++(4*72+18:1cm) coordinate (4) node[below]{E};
			\filldraw[fill=lightgray] (0) -- (1) -- (2) -- (3) -- (4) -- cycle;
			\path (0) ++(-18:1cm) coordinate (B);
			\path (B) ++(1*72+18:1cm) coordinate (5) node[above]{E};
			\path (B) ++(2*72+18:1cm) coordinate (6);
			\path (B) ++(3*72+18:1cm) coordinate (7) node[below]{B};
			\path (B) ++(4*72+18:1cm) coordinate (8) node[below]{C};
			\path (B) ++(5*72+18:1cm) coordinate (9) node[right]{D};
			\filldraw[fill=lightgray] (5) -- (6) -- (7) -- (8) --  (9) -- cycle;
			\draw[->] (0) + (-108:0.5cm) arc (-108:36:0.5cm) node[pos=0.5,right]{$ \frac{4}{5} \pi$};
			\filldraw [black] (0) circle (1.5pt);
			\filldraw [black] (1) circle (1.5pt);
			\filldraw [black] (2) circle (1.5pt);
			\filldraw [black] (3) circle (1.5pt);
			\filldraw [black] (4) circle (1.5pt);
			\filldraw [black] (5) circle (1.5pt);
			\filldraw [black] (6) circle (1.5pt);
			\filldraw [black] (7) circle (1.5pt);
			\filldraw [black] (8) circle (1.5pt);
			\filldraw [black] (9) circle (1.5pt);
			
			%drugi
			\coordinate (C) at (6,0);
			\path (C) ++(18:1cm) coordinate (B0) node[left]{A};
			\path (C) ++(1*72+18:1cm) coordinate (B1) node[above]{B};
			\path (C) ++(2*72+18:1cm) coordinate (B2) node[left]{C};
			\path (C) ++(3*72+18:1cm) coordinate (B3) node[below]{D};
			\path (C) ++(4*72+18:1cm) coordinate (B4) node[below]{E};
			\filldraw[fill=lightgray] (B0) -- (B1) -- (B2) -- (B3) -- (B4) -- cycle;
			
			\path (B0) ++(54:1cm) coordinate (D);
			\path (D) ++(18:1cm) coordinate (C0) node[right]{C};
			\path (D) ++(1*72+18:1cm) coordinate (C1) node[above]{D};
			\path (D) ++(2*72+18:1cm) coordinate (C2) node[above]{E};
			\path (D) ++(3*72+18:1cm) coordinate (C3);
			\path (D) ++(4*72+18:1cm) coordinate (C4) node[below]{B};
			\filldraw[fill=lightgray] (C0) -- (C1) -- (C2) -- (C3) -- (C4) -- cycle;
			\draw[->] (B0) + (-108:0.5cm) arc (-108:108:0.5cm) node[pos=0.7,right]{$ \frac{6}{5} \pi$};
			\filldraw [black] (B0) circle (1.5pt);
			\filldraw [black] (B1) circle (1.5pt);
			\filldraw [black] (B2) circle (1.5pt);
			\filldraw [black] (B3) circle (1.5pt);
			\filldraw [black] (B4) circle (1.5pt);
			\filldraw [black] (C0) circle (1.5pt);
			\filldraw [black] (C1) circle (1.5pt);
			\filldraw [black] (C2) circle (1.5pt);
			\filldraw [black] (C3) circle (1.5pt);
			\filldraw [black] (C4) circle (1.5pt);
			\path (C) ++(18:1cm) coordinate (B0) node[left]{A};
		\end{tikzpicture}
		\caption{Rotation of 0-complex by two possible angles when $ k = 5$.}\label{obrót}
		\end{figure}
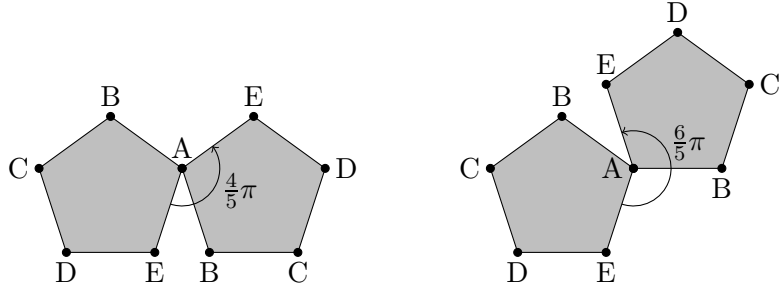
	\end{center}
	\begin{theorem}
	\label{thm:odd}
		For $\mathcal{K}^{\langle\infty\rangle}$ with an odd number of essential fixed points $k$, the following conditions are equivalent. 
		\begin{itemize}
			\item[(A)] The fractal $\mathcal{K}^{\langle\infty\rangle}$ has GLP.
			\item[(B)] Any cycle of complexes $\Delta_0$ in $\mathcal{K}^{\langle 1\rangle}$, consisting of $c$ rotations by the angle $\frac{k+1}{k}\pi$ and $d$ times by the angle $\frac{k-1}{k}\pi$ satisfies $k|(c-d)$. 
		\end{itemize} 
	\end{theorem}
	\begin{proof}
		$(A)\Rightarrow(B)$ Since $\mathcal{K}^{\langle\infty\rangle}$ has GLP, then of course $\mathcal{K}^{\langle1\rangle}$ and any its subset can be well labelled. Hence, any cycle of rotation of labels "going back" to the starting 0-complex must coincide with it, otherwise $ \mathcal{K}^{\langle \infty \rangle}$ would not have GLP.
		
		That is, for any cycle with a fixed number of rotations $ c $ and $ d $, there is an $ n $ such that holds
		$$
		c\frac{k+1}{k}\pi+d\frac{k-1}{k}\pi=2\pi n.
		$$
		After few simple transformations, we get
		$$
		c+d+\frac{c-d}{k}=2n.
		$$
		$(c+d)$ and $ 2n $ are integers, therefore the left side is also an integer if and only if  $k|(c-d)$.\\
		 $(A)\Leftarrow(B)$ Suppose the following condition holds for a given cycle of complexes
		$$
		c\frac{k+1}{k}\pi+d\frac{k-1}{k}\pi=2\pi n.
		$$
		This cycle can be well labeled - we label vertices of the complex $\Delta_{0,1}$ and then recursively having labeled vertices of $\Delta_{0,i}$ for some $ i $ we can rotate the labels to $\Delta_{0,i+1}$ analogously to the rotation of the whole complex. Thanks to the assumption, we know that the labels after the rotation of $\Delta_{0,j}$ on $\Delta_{0,1}$ will match those originally assigned to the vertices of $\Delta_{0,1}$. Thus, indeed the complex vertices in the cycle can be well labeled.
		
		In the next step, we label the vertices of the remaining complexes as follows: select the complex $\Delta_0 \subset \mathcal{K}^{\langle 1\rangle}$ with unlabeled vertices, which is adjacent to some labeled complex $\Delta_0^{'}$. $\Delta_0$ is created by rotating $\Delta_0^{'}$ by the angle $\frac{k+1}{k} \pi$ or $\frac{k-1}{k} \pi$ and analogously to this rotation, we rewrite the labels from $\Delta_0^{'}$ on $\Delta_0$. Meeting the assumption guarantees that each vertex will receive exactly one label (because whole cycles can be labeled in one way). Moreover, since the labels on each cell are arranged in the same orientation, the condition of good labeling properties will be met, that is the fractal has GLP.
	\end{proof}

	\subsection{Reduction of the test area for GLP verification}
	For larger number of 0-complexes in $\mathcal{K}^{\langle 1\rangle}$ the verification of GLP on $\mathcal{K}^{\langle 1\rangle}$ using the characterization theorems can be time-consuming. We will show that the area of GLP verification can be reduced to "two $ k $-th" of $\mathcal{K}^{\langle 1\rangle}$.

	\begin{definition}
		Let $\mathcal{K}^{\langle 1\rangle}$ have $k$ essential fixed points.
		The bounded area $U_i$ of the fractal $\mathcal{K}^{\langle 1\rangle}$ is the sum of sets
		\begin{itemize}
		\item the angle (without rays) of measure $\frac{2\pi}{k}$ and vertex $ \mathcal{B}_{\mathcal{K}^{\langle 1\rangle}} $, where rays are the half-lines included in axes of symmetry of $\mathcal{K}^{\langle 1\rangle}$ which pass through vertices of the complex,
		\item half-line without the initial point being the left edge (from $ \mathcal{B}_{\mathcal{K}^{\langle 1\rangle}}$ point of view) of the angle above, that is half of one of two axes of symmetry bounding that angle
		\end{itemize}
		intersected with $\mathcal{K}^{\langle 1 \rangle}$.

		The area $U_i$ does not contain $\mathcal{B}_{\mathcal{K}^{\langle 1\rangle}}$, therefore $\bigcup_{j=1}^{k} U_j= \mathcal{K}^{\langle 1\rangle}\backslash \mathcal{B}_{\mathcal{K}^{\langle 1\rangle}}$ and that sum is disjoint.
	
	The fractal slice $W_i$ is defined as the sum of those cells $\Delta_0 \subset \mathcal{K}^{\langle 1 \rangle}$ which centers $\mathcal{B}_{\Delta_0}$ lie in the area $ U_i $.
	\end{definition}
	\begin{figure}[H]
	\begin{center}
	\begin{tikzpicture}[scale=2]
		%7
		\coordinate (A) at (0,0);
		%tworzymy siedmiokąt foremny
		\path (A) ++(0:1cm) coordinate (0);
		\path (0) ++(1*2*180/7:1cm) coordinate (1);
		\path (1) ++(2*2*180/7:1cm) coordinate (2) ;
		\path (2) ++(3*2*180/7:1cm) coordinate (3);
		\path (3) ++(4*2*180/7:1cm) coordinate (4) ;
		\path (4) ++(5*2*180/7:1cm) coordinate (5) ;
		\path (5) ++(6*2*180/7:1cm) coordinate (6) ;
		
		\path (A) ++(5*180/14:1.152382435481243252620575111773427556722250943803160581031553148cm) coordinate (B);%{$ \mathcal{B}_{\mathcal{K}^{\langle 1\rangle}} $};
		\path (B) ++(90:1.7cm) coordinate (B1);
		\path (B) ++(90+ 1*180/7:1.7cm) coordinate (B2);
		\path (B) ++(90+ 2*180/7:1.7cm) coordinate (B3);
		\path (B) ++(90+ 3*180/7:1.7cm) coordinate (B4);
		\path (B) ++(90+ 4*180/7:1.7cm) coordinate (B5);
		\path (B) ++(90+ 5*180/7:1.7cm) coordinate (B6);
		\path (B) ++(90+ 6*180/7:1.7cm) coordinate (B7);
		\path (B) ++(90+ 7*180/7:1.7cm) coordinate (B8) node[below]{$U_1$};
		\path (B) ++(90+ 8*180/7:1.7cm) coordinate (B9);
		\path (B) ++(90+ 9*180/7:1.7cm) coordinate (B10);
		\path (B) ++(90+ 10*180/7:1.7cm) coordinate (B11);
		\path (B) ++(90+ 11*180/7:1.7cm) coordinate (B12);
		\path (B) ++(90+ 12*180/7:1.7cm) coordinate (B13);
		\path (B) ++(90+ 13*180/7:1.7cm) coordinate (B14);
		
		%wierzchołki
		\foreach \x in {0,1,2,3,4,5,6}
		{
			\filldraw [black] (\x) circle (1pt);
		}

		\draw[color=orange] (B9)--(B);
		\filldraw[fill=orange!20,draw=orange!20] (B) -- (B7) arc (90+ 6*180/7:90+ 8*180/7:1.7cm) -- cycle; 
		\draw[orange] (B9)--(B);
		\draw[dashed] (B9)--(B);
		
		\draw[dashed] (B1)--(B);
		\draw[dashed] (B2)--(B);
		\draw[dashed] (B3)--(B);
		\draw[dashed] (B4)--(B);
		\draw[dashed] (B5)--(B);
		\draw[dashed] (B6)--(B);
		\draw[white] (B7)--(B);
		\draw[dashed] (B7)--(B);
		\draw[dashed] (B8)--(B);
		\draw[dashed] (B10)--(B);
		\draw[dashed] (B11)--(B);
		\draw[dashed] (B12)--(B);
		\draw[dashed] (B13)--(B);
		\draw[dashed] (B14)--(B);
		\draw (0) -- (1) -- (2) -- (3) -- (4) -- (5) -- (6) --cycle;
		\filldraw[fill=white,draw=black] (B)++(.4mm,0) arc (0:360:0.4mm) ;
		\node[right,yshift=0.4cm] at (B) {\contour{white}{$ \mathcal{B}_{\mathcal{K}^{\langle 1\rangle}} $}};
	\end{tikzpicture}
	\end{center}
	\caption{Area $ U_1 $ in $ \mathcal{K}^{\langle 1\rangle} $ for $ k$ equal 7.}\label{fig:obszarui}
	\end{figure}
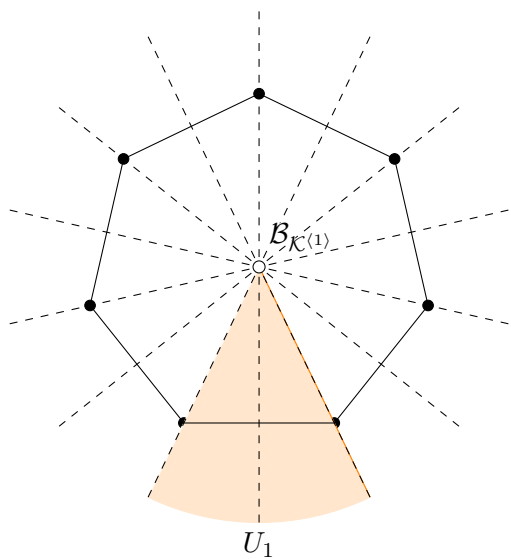
	
	It can be shown that $ \bigcup_{j = 1}^{k}W_j = \mathcal{K}^{\langle 1 \rangle} $ holds for $ k> 4 $ and $ k \neq 6 $, because in this case $ \mathcal{B}_{\mathcal{K}^{\langle 1 \rangle}} $ does not belong to $ \mathcal{K}^{\langle 1 \rangle} $. This fact results from the Lemma \ref{D0 w BK1} in the further part of the work.

	\begin{theorem}\label{tw:niep}
	If 
	\begin{enumerate}
	\item[1)] $k\neq 6$

	or

	\item[2)] $k=6$, but $\mathcal{B}_{\mathcal{K}^{\langle 1\rangle}} \notin \mathcal{K}^{\langle 1\rangle} $,
	\end{enumerate}
	then $ \mathcal{K}^{\langle \infty \rangle} $ has GLP if and only if it is possible to well label vertices in two neighboring slices $ W_i $, $ W_{i+1}$ of the fractal $ \mathcal{K}^{\langle 1 \rangle}. $
	\end{theorem}
	For $ k = 6 $ we can construct $ \mathcal{K}^{\langle 1 \rangle} $ such that for some 0-complex $ \Delta_0 $ its center $ \mathcal {B}_{\Delta_ {0 }} $ would be the point $ \mathcal{B}_{\mathcal{K}^{\langle 1 \rangle}}$, which would mean that the selected $\Delta_{0} $ would not be in any slice $ W_i $.
	In this case, $ \mathcal{K}^{\langle \infty \rangle} $ has no GLP, because the connectivity condition forces the 0-complex in $ \mathcal{B}_{\mathcal{K}^{ \langle 1 \rangle}} $ to be adjacent to some other, and the symmetry condition forces it to have exactly six neighbors.
	Using the Theorem \ref {thm:even} we show that $\mathcal{K}^{\langle 1 \rangle} $ cannot be well labeled. The Figure \ref{ziel} shows our reasoning.
	  \begin{figure}[H]
	  	\centering
	  	\begin{tikzpicture}[scale=0.93]
	  		%6
	  		\coordinate (00) at (0,0);
	  		%tworzymy sześciokąt foremny
	  		\path (00)++(-120:2cm) coordinate(AA);
	  		\foreach \x in {0,1,2,3,4,5}
	  		{
	  			\path (00) ++(\x*60 :1cm) coordinate (\x);
	  		}
	  		\filldraw[fill=lightgray,draw=black] (0) -- (1) -- (2) -- (3) -- (4) -- (5) -- cycle;
  		
	  		\foreach \x in {0,1,2,3,4,5}
	  		{
	  			\path (AA) ++(\x*60 :1cm) coordinate (A\x);
	  		}
	  		\filldraw[fill=brown,draw=black] (A0) -- (A1) -- (A2) -- (A3) -- (A4) -- (A5) -- cycle;
	  		
	  		\foreach \x in {0,1,2,3,4,5}
	  		{
	  			\filldraw [black] (\x) circle (1.5pt);
	  		}
  			\foreach \x in {0,2,3,4,5}
  			{
  				\filldraw [black] (A\x) circle (1.5pt);
  			}
  			
  			\foreach \x in {0,1,2,3,...,11}
  			{
  				\draw[dashed] (00)--(\x*30:3.5cm);
  			}
  			\node[right,yshift=0.2cm,xshift=-0.05cm, xscale=0.8, yscale=0.8] at (00) {\contour{lightgray}{$ \mathcal{B}_{\mathcal{K}^{\langle 1\rangle}} $}};
  			\filldraw [black] (00) circle (2.5pt);
  		
	  		%6 nr. 2
	  		\coordinate (1A) at (8,-sqrt(3);
	  		\path (1A) ++(0:1cm) ++(180:1cm) node{1};
	  		%tworzymy sześciokąt foremny
	  		\path (1A) ++(0:1cm) coordinate (10);
	  		\path (1A) ++(1*60:1cm) coordinate (11);
	  		\path (1A) ++(2*60:1cm) coordinate (12) ;
	  		\path (1A) ++(3*60:1cm) coordinate (13);
	  		\path (1A) ++(4*60:1cm) coordinate (14) ;
	  		\path (1A) ++(5*60:1cm) coordinate (15);
	  		\filldraw[fill=brown,draw=black] (10) -- (11) -- (12) -- (13) -- (14) -- (15) -- cycle;
	  		\path (1A) ++(0:1cm) ++(180:1cm) node[left]{1};
	  		\path (1A) ++(2*60:1cm) coordinate (12) ;
	  		%drugi
	  		\path (1A) ++(0:2cm) coordinate (1B) ;
	  		\path (1B) ++(0:1cm) coordinate (1B0) ;
	  		\path (1B) ++(1*60:1cm) coordinate (1B1) ;
	  		\path (1B) ++(2*60:1cm) coordinate (1B2);
	  		\path (1B) ++(3*60:1cm) coordinate (1B3);
	  		\path (1B) ++(4*60:1cm) coordinate (1B4) ;
	  		\path (1B) ++(5*60:1cm) coordinate (1B5);
	  		\filldraw[fill=lightgray,draw=black] (1B0) -- (1B1) -- (1B2) -- (1B3) -- (1B4) -- (1B5) -- cycle;
	  		\path (1B) ++(1*60:1cm) coordinate (1B1) ;
	  		\path (1A) ++(0:2cm) coordinate (1B)  node[right]{2};
	  		%trzeci
	  		\path (1B) ++(60:2cm) coordinate (1C);
	  		\path (1C) ++(0:1cm) coordinate (1C0);
	  		\path (1C) ++(1*60:1cm) coordinate (1C1) ;
	  		\path (1C) ++(2*60:1cm) coordinate (1C2) ;
	  		\path (1C) ++(3*60:1cm) coordinate (1C3)  ;
	  		\path (1C) ++(4*60:1cm) coordinate (1C4);
	  		\path (1C) ++(5*60:1cm) coordinate (1C5) ;
	  		\filldraw[fill=lightgray,draw=black] (1C0) -- (1C1) -- (1C2) -- (1C3) -- (1C4) -- (1C5) -- cycle;
	  		\path (1B) ++(60:2cm) coordinate (1C) node[above]{1};
	  		%czwarty
	  		\path (1C) ++(120:2cm) coordinate (1D) ;
	  		\path (1D) ++(0:1cm) coordinate (1D0) ;
	  		\path (1D) ++(1*60:1cm) coordinate (1D1) ;
	  		\path (1D) ++(2*60:1cm) coordinate (1D2) ;
	  		\path (1D) ++(3*60:1cm) coordinate (1D3)  ;
	  		\path (1D) ++(4*60:1cm) coordinate (1D4) ;
	  		\path (1D) ++(5*60:1cm) coordinate (1D5);
	  		\filldraw[fill=lightgray,draw=black] (1D0) -- (1D1) -- (1D2) -- (1D3) -- (1D4) -- (1D5) -- cycle;
	  		\path (1C) ++(120:2cm) coordinate (1D) node[left]{2};
	  		%piąty
	  		\path (1D) ++(180:2cm) coordinate (1E) ;
	  		\path (1E) ++(0:1cm) coordinate (1E0);
	  		\path (1E) ++(1*60:1cm) coordinate (1E1) ;
	  		\path (1E) ++(2*60:1cm) coordinate (1E2) ;
	  		\path (1E) ++(3*60:1cm) coordinate (1E3);
	  		\path (1E) ++(4*60:1cm) coordinate (1E4) ;
	  		\path (1E) ++(5*60:1cm) coordinate (1E5) ;
	  		\filldraw[fill=lightgray,draw=black] (1E0) -- (1E1) -- (1E2) -- (1E3) -- (1E4) -- (1E5) -- cycle;
	  		\path (1D) ++(180:2cm) coordinate (1E) node[right]{1};
	  		%szósty
	  		\path (1E) ++(240:2cm) coordinate (1F) ;
	  		\path (1F) ++(0:1cm) coordinate (1F0) ;
	  		\path (1F) ++(1*60:1cm) coordinate (1F1);
	  		\path (1F) ++(2*60:1cm) coordinate (1F2) ;
	  		\path (1F) ++(3*60:1cm) coordinate (1F3)  ;
	  		\path (1F) ++(4*60:1cm) coordinate (1F4) ;
	  		\path (1F) ++(5*60:1cm) coordinate (1F5);
	  		\filldraw[fill=lightgray,draw=black] (1F0) -- (1F1) -- (1F2) -- (1F3) -- (1F4) -- (1F5) -- cycle;
	  		\path (1E) ++(240:2cm) coordinate (1F) node[above]{2};
	  		%siódmy
	  		\path (1F) ++(0:2cm) coordinate (1G);
	  		\path (1G) ++(0:1cm) coordinate (1G0);
	  		\path (1G) ++(1*60:1cm) coordinate (1G1) ;
	  		\path (1G) ++(2*60:1cm) coordinate (1G2) ;
	  		\path (1G) ++(3*60:1cm) coordinate (1G3) ;
	  		\path (1G) ++(4*60:1cm) coordinate (1G4) ;
	  		\path (1G) ++(5*60:1cm) coordinate (1G5) ;
	  		\filldraw[fill=lightgray,draw=black] (1G0) -- (1G1) -- (1G2) -- (1G3) -- (1G4) -- (1G5) -- cycle;
	  		\path (1F) ++(0:2cm) coordinate (1G) node{?};	
	  		\path (1B) ++(2*60:1cm) coordinate (1B2) ;

	  		\foreach \x in {0,1,2,3,...,11}
	  		{
	  			\draw[dashed] (1G)--+(\x*30:3.5cm);
	  		}
  			
  			\filldraw [lightgray] (9,0) circle (0.15cm);
  			%siódmy
  			\path (1F) ++(0:2cm) coordinate (1G);
  			\path (1G) ++(0:1cm) coordinate (1G0);
  			\path (1G) ++(1*60:1cm) coordinate (1G1) ;
  			\path (1G) ++(2*60:1cm) coordinate (1G2) ;
  			\path (1G) ++(3*60:1cm) coordinate (1G3) ;
  			\path (1G) ++(4*60:1cm) coordinate (1G4) ;
  			\path (1G) ++(5*60:1cm) coordinate (1G5) ;
  			\draw[draw=black] (1G0) -- (1G1) -- (1G2) -- (1G3) -- (1G4) -- (1G5) -- cycle;
  			\path (1F) ++(0:2cm) coordinate (1G) node{?};	
  			\path (1B) ++(2*60:1cm) coordinate (1B2) ;
	  		%wierzchołki
	  		\filldraw [black] (10) circle (1.5pt);
	  		\filldraw [black] (11) circle (1.5pt);
	  		\filldraw [black] (12) circle (1.5pt);
	  		\filldraw [black] (13) circle (1.5pt);
	  		\filldraw [black] (14) circle (1.5pt);
	  		\filldraw [black] (15) circle (1.5pt);
	  		\filldraw [black] (1B1) circle (1.5pt);
	  		\filldraw [black] (1B2) circle (1.5pt);
	  		\filldraw [black] (1B3) circle (1.5pt);
	  		\filldraw [black] (1B4) circle (1.5pt);
	  		\filldraw [black] (1B0) circle (1.5pt);
	  		\filldraw [black] (1B5) circle (1.5pt);
	  		\filldraw [black] (1C0) circle (1.5pt);
	  		\filldraw [black] (1C1) circle (1.5pt);
	  		\filldraw [black] (1C2) circle (1.5pt);
	  		\filldraw [black] (1C3) circle (1.5pt);
	  		\filldraw [black] (1C4) circle (1.5pt);
	  		\filldraw [black] (1C5) circle (1.5pt);
	  		\filldraw [black] (1D0) circle (1.5pt);
	  		\filldraw [black] (1D1) circle (1.5pt);
	  		\filldraw [black] (1D2) circle (1.5pt);
	  		\filldraw [black] (1D3) circle (1.5pt);
	  		\filldraw [black] (1D4) circle (1.5pt);
	  		\filldraw [black] (1D5) circle (1.5pt);
	  		\filldraw [black] (1E0) circle (1.5pt);
	  		\filldraw [black] (1E1) circle (1.5pt);
	  		\filldraw [black] (1E2) circle (1.5pt);
	  		\filldraw [black] (1E3) circle (1.5pt);
	  		\filldraw [black] (1E4) circle (1.5pt);
	  		\filldraw [black] (1E5) circle (1.5pt);
	  		\filldraw [black] (1F0) circle (1.5pt);
	  		\filldraw [black] (1F1) circle (1.5pt);
	  		\filldraw [black] (1F2) circle (1.5pt);
	  		\filldraw [black] (1F3) circle (1.5pt);
	  		\filldraw [black] (1F4) circle (1.5pt);
	  		\filldraw [black] (1F5) circle (1.5pt);					
	  	\end{tikzpicture}
  		\caption{Attaching a brown $ 0 $-complex to the one in $ \mathcal{B}_{\mathcal{K}^{\langle 1 \rangle}} $, in the left Figure, forces $\mathcal{K}^{\langle 1 \rangle} $ to be constructed as on the right.}\label{ziel}
  		\end{figure}
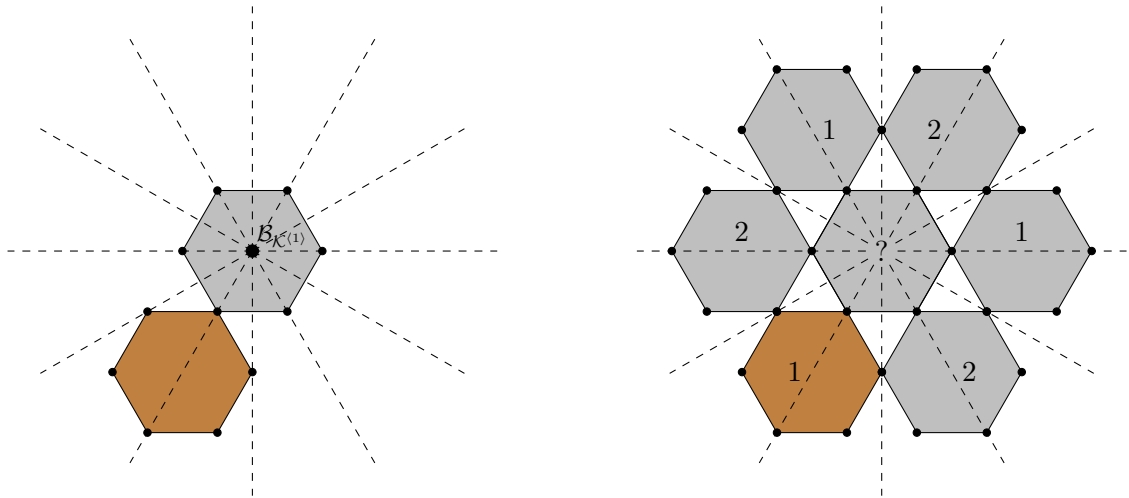
  	Before we move on to the proof of Theorem \ref {tw:niep}, we will prove the following lemmas.

	\begin{lemma}\label{lem:niezmienniczosc na obroty}
		Let $ \mathcal{K}^{\langle 1 \rangle} $ be a fractal with $ k $ essential fixed points. Then it is invariant under rotation by angle $ \frac{2 \pi}{k} $.
	\end{lemma}

	\begin{proof}
		It is enough to show that any point belonging to $ \mathcal{K}^{\langle 1 \rangle} $ after rotation by the angle $ \frac {2 \pi} {k} $ coincides with another. Let us take the point $x \in \mathcal{K}^{\langle 1 \rangle}$ being at distance $r$ from $ \mathcal{B}_{\mathcal{K}^{\langle 1 \rangle}}$, at distance $t_1$ from the axis of symmetry to the right of $x$ looking from $\mathcal{B}_{\mathcal{K}^{\langle 1 \rangle}}$ and $t_2$ from the axis of symmetry on the left. We will denote these axes of symmetry $L_1$ and $L_2$ respectively, and the next one lying to the left of $L_2$ will be denoted $L_3$ (as in the Figure \ref{obroty punktow}).
		Due to the symmetry condition, there is a point inside the fractal symmetric to $x$ with respect to $L_2$. Let us call it $y$. Just like $x$, it is at distance $t_2$ from $L_2$ and at distance $r$ from $ \mathcal{B}_{\mathcal{K}^{\langle 1 \rangle}}$. Moreover, since the image of $L_1$ in symmetry with respect to $L_2$ is $L_3$, then the distance of $y$ to $L_3$ is equal to $t_1$.
		Now let us take a look at the point symmetric to $y$ with respect to $L_3$, let us call it $z$. It is distant from $ \mathcal{B}_{\mathcal{K}^{\langle 1 \rangle}}$ by $r$ and distant from $L_3$ by $t_1$.
		The angle between consecutive fractal symmetry axes is $\frac{\pi}{k}$, so the angle between $L_1$ and $L_3$ is $\frac{2\pi}{k}$. The same is the angle between $x \mathcal{B}_{\mathcal{K}^{\langle 1 \rangle}}$ and $z \mathcal{B}_{\mathcal{K}^{\langle 1 \rangle}}$, because both segments form the same angle $ \gamma $ with the neighboring axes (here $ L_1 $ and $ L_3 $).
		So $z \in \mathcal{K}^{\langle 1 \rangle}$ is the image of the point $x$ by rotation around $\mathcal{B}_{\mathcal{K}^{\langle 1 \rangle}}$ by angle $\frac {2 \pi}{k}$.
	\end{proof}

		\begin{figure}[H]
			\centering
			\begin{tikzpicture}[scale=1.25]
				\draw[draw=black] (0,0) -- (-2.7cm,0) arc (-180:-90+atan(75/325):2.7cm);
				\fill[fill=white] (0,0) -- (-3.1cm,0) arc (180:180+atan(75/325):3.1cm) -- cycle;
				\fill[fill=white] (0,0) -- (-3.1cm,0) arc (180-atan(75/325):180:3.1cm) -- cycle;
				\draw[white,line width=10] (-4,0)--(4,0);
				
				%symetrie
				\draw[dashed] (0.5,0)--(-4,0);
				\node at (-4.2,0) {$ L_1 $};
				
				\draw[dashed] (0,-4)--(0,0.5);
				\node at (0,-4.27) {$ L_3 $};
				\coordinate (00) at (0,0) node[right, xshift=2.7mm,yshift=2.3mm]{$\mathcal{B}_{\mathcal{K}^{\langle 1\rangle}}$ };
				\path (00) ++(45:0.5cm) coordinate (01);
				\path (00) ++(45+180:4cm) coordinate (02);
				\path (00) ++(45+90:0.5cm) coordinate (03);
				\path (00) ++(-45:0.5cm) coordinate (04);
				\draw[dashed] (01)--(02);
				\draw[dashed] (03)--(04);
				\node[xshift=-1.7mm,yshift=-1.7mm] at (02) {$ L_2 $};
				%punkt x
				\coordinate (0x) at (-3.25,0);
				\path (0x) ++(0,-0.75) coordinate(x);
				%ozn. t_1, r, x
				\draw(0x)--(x) node[left,pos=0.5]{$ t_1 $};
				\draw(x)--(00) node[below,pos=0.5]{$ r $};
				\filldraw [black] (x) circle (1pt) node[left]{$ x $};
				%kąt $ \frac{2\pi}{k} $
				\filldraw[fill=white,draw=black] (0,0) -- (-1cm,0) arc (180:270:1cm) -- cycle;
				\node at (-0.5,-0.5) {$ \frac{2\pi}{k} $};
				%kąt gamma 
				\node at (-2.2,-0.25) {$ \gamma $};
				\draw[draw=black] (0,0) -- (-2.5cm,0) arc (180:180+atan(75/325):2.5cm) -- cycle;
				%punkt y i z
				\coordinate (0y) at (0,-3.25);
				\path (0y) ++(-0.75,0) coordinate(y);
				\path (0y) ++(0.75,0) coordinate(z);
				%ozn. t_1, r, y, z
				\draw(0y)--(y) node[below,pos=0.5]{$ t_1 $};
				\draw(y)--(00) node[left,pos=0.5]{$ r $};
				\filldraw [black] (y) circle (1pt) node[below,xshift=-0.5mm]{$ y $};
				\draw(0y)--(z) node[below,pos=0.5]{$ t_1 $};
				\filldraw [black] (0y) circle (0.75pt);
				\filldraw [black] (0x) circle (0.75pt);
				\draw(z)--(00) node[right,pos=0.5]{$ r $};
				\filldraw [black] (z) circle (1pt) node[below,xshift=0.5mm]{$ z $};
				\filldraw [black] (0,0) circle (2.3pt);
				%kąt gamma 
				\node at (0.25,-2.2) {$ \gamma $};
				\draw[draw=black] (0,0) -- (0,-2.5cm) arc (270:270+atan(75/325):2.5cm) -- cycle;
				
				%odcinek z x do y
				\draw[draw=black] (x) -- (y)  node[pos=1/4, xshift=-5,yshift=-5]{$ t_2 $} node[pos=3/4, xshift=-5,yshift=-5]{$ t_2 $};
				\node at (-1.2,-2.1) {$ \frac{2\pi}{k} $};
				\filldraw [black] (x) ++(-45:1.77cm) circle (0.75pt);
			\end{tikzpicture}
			\caption{Illustration for the Lemma \ref{lem:niezmienniczosc na obroty} proof.}\label{obroty punktow}
		\end{figure}
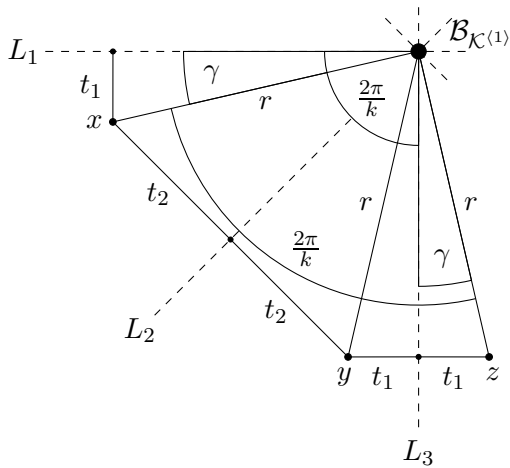
		
	The next two lemmas give us key information about the position of 0-complexes in $\mathcal{K}^{\langle 1 \rangle} $.

	\begin{lemma}\label{a}
		If $ k> 3 $, then there is no 0-complex $\Delta_{0}$ such that any of its vertices lies at the meeting point of all $U_i$ areas, i.e. in $ \mathcal{B}_{\mathcal{K}^{\langle 1 \rangle}}$.
	\end{lemma}
	\begin{proof}
		
		Since $\mathcal{K}^{\langle 1 \rangle}$ has $k$ axes of symmetry, the smaller of the angles between consecutive axes is $\frac{\pi}{k}$.
		
		It is known that the angle in the convex hull $\mathcal{H}(\Delta_0)$, which is a regular polygon with $k$ vertices, is equal to $\frac{k-2}{k}\pi$, for $k \in \{4,5,6, \dots \}$ (by assumption we exclude equilateral triangles).
		
		If any sides of $\mathcal {H}(\Delta_0)$ were on the axes of symmetry of $\mathcal{K}^{\langle 1 \rangle}$, we would find two neighbor 0-complexes with common sides (but we know that they can only have joint vertices). 
		Similarly, if the interior of $\mathcal{H}(\Delta_0)$ is crossed by the axis of symmetry of $\mathcal{K}^{\langle 1 \rangle}$, then that axis must also be the axis of symmetry of $\mathcal{H}(\Delta_0)$ (that is the bisector of the angle of the polygon $\mathcal{H}(\Delta_0)$ at the vertex $\mathcal{B} _{\mathcal{K}^{\langle 1 \rangle}} $).
		
		Thus, $\mathcal{H}(\Delta_0) $ can be intersected by at most one axis of symmetry of $\mathcal{K}^{\langle 1 \rangle}$ and the internal angle of $\mathcal{H}(\Delta_0) $ must be less than $\frac{2\pi}{k}$.
		
		It means that the following inequality holds
		$$
		\frac{k-2}{k}\pi<\frac{2\pi}{k},
		$$
		and equivalently $k<4$.
	\end{proof}

		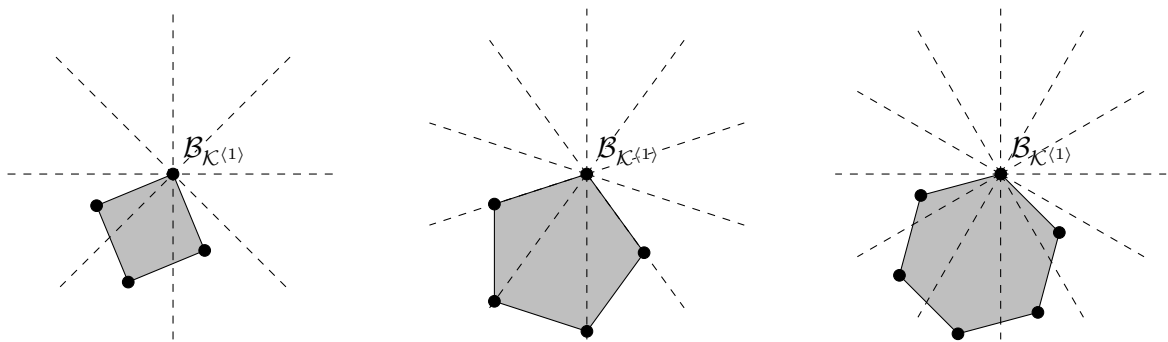
\begin{figure}[H]
		\begin{center}
		\begin{tikzpicture}[scale=1.1]
			%4
			\coordinate (00) at (0,0);
			\path (00) ++(45:2cm) coordinate (01);
			\path (00) ++(45+180:2cm) coordinate (02);
			\path (00) ++(45+90:2cm) coordinate (03);
			\path (00) ++(-45:2cm) coordinate (04);
			\path (00) ++(180+22.5:1cm) coordinate (101);
			\path (101) ++(-45-22.5:1cm) coordinate (102);
			\path (102) ++(22.5:1cm) coordinate (103);
			\filldraw[fill=lightgray,draw=black] (101) -- (102) -- (103) -- (00) -- cycle;
			\draw[dashed] (-2,0)--(2,0);
			\draw[dashed] (0,-2)--(0,2);
			\draw[dashed] (01)--(02);
			\draw[dashed] (03)--(04);
			\filldraw [black] (0,0) circle (2pt);
			\filldraw [black] (101) circle (2pt);
			\filldraw [black] (102) circle (2pt);
			\filldraw [black] (103) circle (2pt);
			%5
			\coordinate (A) at (5,0);
			%tworzymy pięciokąt foremny
			\path (A) ++(3*72+36:1cm) coordinate(A2);
			\path (A2) ++(0:1cm) coordinate (0);
			\path (A2) ++(1*72:1cm) coordinate (1);
			\path (A2) ++(2*72:1cm) coordinate (2);
			\path (A2) ++(3*72:1cm) coordinate (3);
			\path (A2) ++(4*72:1cm) coordinate (4);
			\filldraw[fill=lightgray,draw=black] (0) -- (1) -- (2) -- (3) -- (4) -- cycle;
			\foreach \x in{0,1,2,3,...,9}{
				\path (A) ++(-18+\x*36:2cm) coordinate (P\x);
				\draw[dashed] (A)--(P\x);
			}
			\filldraw [black] (5,0) circle (2pt);
			\filldraw [black] (0) circle (2pt);
			\filldraw [black] (4) circle (2pt);
			\filldraw [black] (2) circle (2pt);
			\filldraw [black] (3) circle (2pt);
			
			%6
			\coordinate (B) at (10,0);
			\path (B) ++(-30:2cm) coordinate (P0b);
			\path (B) ++(30:2cm) coordinate (P1b);
			\path (B) ++(90:2cm) coordinate (P2b);
			\path (B) ++(150:2cm) coordinate (P3b);
			\path (B) ++(210:2cm) coordinate (P4b);
			\path (B) ++(270:2cm) coordinate (P5b);
			%tworzymy sześciokąt foremny
			\path (B) ++(255:1cm) coordinate(B1);
			\path (B1) ++(15:1cm) coordinate (0b);
			\path (B1) ++(1*60+15:1cm) coordinate (1b);
			\path (B1) ++(2*60+15:1cm) coordinate (2b);
			\path (B1) ++(3*60+15:1cm) coordinate (3b);
			\path (B1) ++(4*60+15:1cm) coordinate (4b);
			\path (B1) ++(5*60+15:1cm) coordinate (5b);
			\filldraw[fill=lightgray,draw=black]  (0b) -- (1b) -- (2b) -- (3b) -- (4b) --(5b)-- cycle;
			
			\draw[dashed] (B)--(P0b);
			\draw[dashed] (B)--(P1b);
			\draw[dashed] (B)--(P2b);
			\draw[dashed] (B)--(P3b);
			\draw[dashed] (B)--(P4b);
			\draw[dashed] (B)--(P5b);
			\foreach \x in {1,2,3,...,6}{
				\path (B) ++(\x *60-60:2cm) coordinate (a\x );
				\draw[dashed] (B)--(a\x);
			}
			\foreach \x in {0,1,2,3,...,5}{
				\filldraw [black] (\x b) circle (2pt);
			}
			\node[right,yshift=0.3cm] at (00) {\contour{white}{$ \mathcal{B}_{\mathcal{K}^{\langle 1\rangle}} $}};
			\node[right,yshift=0.3cm] at (A) {\contour{white}{$ \mathcal{B}_{\mathcal{K}^{\langle 1\rangle}} $}};
			\node[right,yshift=0.3cm] at (B) {\contour{white}{$ \mathcal{B}_{\mathcal{K}^{\langle 1\rangle}} $}};
		\end{tikzpicture}
		\caption{Example of overlapping $ \mathcal{H}(\Delta_0)$ on other axes, looking from the left for $k$ equal to $4$, $5$ and $6$, respectively.}
		\end{center}
		\end{figure}
	
	Now we will show that for almost all values of $k$ a 0-complex cannot be located in the $\mathcal{K}^{\langle 1 \rangle}$ barycenter.

	\begin{lemma}\label{D0 w BK1}
		For $ k>4 $ and $ k \neq 6 $ we have $\mathcal{B}_{\mathcal{K}^{\langle 1\rangle}} \notin \mathcal{K}^{\langle 1\rangle}$.
	\end{lemma} 
	\begin{proof}
		On the contrary, let us suppose that there exists $\Delta_{0} \subset \mathcal{K}^{\langle 1 \rangle}$ containing $\mathcal{B}_{\mathcal{K}^{\langle 1 \rangle}}$. Let us denote it by $\Delta_{\mathcal {B}}$. We know from Lemma \ref{a} that $ \mathcal{B}_{\mathcal {K}^{\langle 1 \rangle}} $ cannot be a vertex of $ \Delta_{\mathcal {B}}$. The symmetry condition forces $\mathcal{B}_{\mathcal{K}^{\langle 1 \rangle}} $ to be also the barycenter of $ \Delta_{\mathcal {B}}$.

		We will show that such $\Delta_{\mathcal {B}}$ cannot be adjacent to any other $\Delta_{0}$. At first we will prove this for an odd $ k>4$, and later for even $ k>6 $ .

		For odd $k$ note that any 0-complex adjacent to $\Delta_{\mathcal {B}}$ must overlap it after rotation by the angle $\frac{k-1}{k}\pi$ or $\frac{k+1}{k}\pi$ around the common vertex (cf. Fact 3.1 in \cite{bib:KOPP}). Then the axis of symmetry of $\mathcal{K}^{\langle 1 \rangle}$, passing through the common vertex of $\Delta_{0}$ and $\Delta_{\mathcal{B}}$, will intersect $ \Delta_{0} $, but it will not be $ \Delta_{0}$'s axis of symmetry.

		The axis of symmetry of $\mathcal{K}^{\langle 1 \rangle}$ would be $\Delta_0 $'s axis of symmetry only when $\Delta_0$ originated from the rotation of $\Delta_{\mathcal {B}}$ by angle $\pi$, which is not allowed for odd $k$ (see Figure \ref{5i5}).

		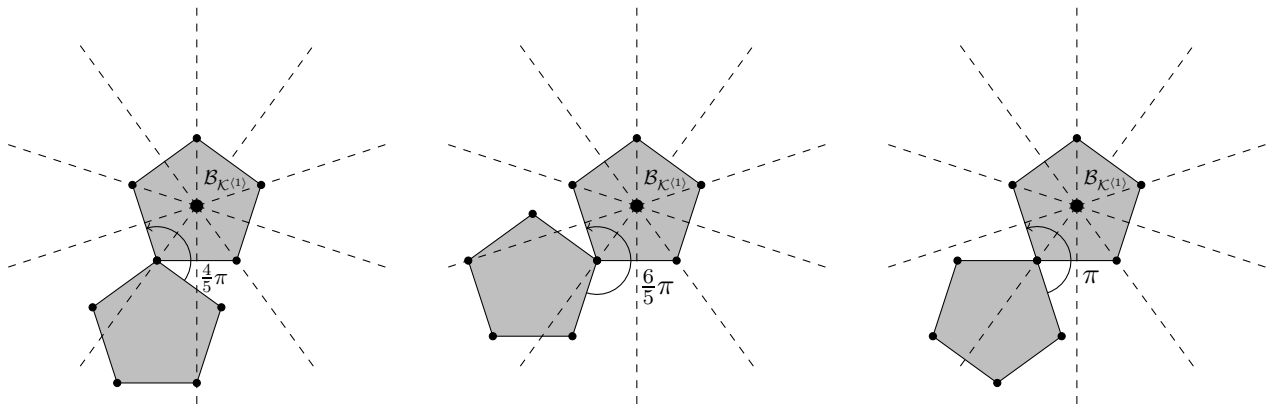
\begin{figure}[H]
		\centering
		\begin{tikzpicture}[scale=0.9]
			\coordinate (A) at (0,0);
			\foreach \x in {1,2,3,...,5}{
				\path (A) ++(\x*72+18:1cm) coordinate (a\x);
			}
			\filldraw[fill=lightgray,draw=black] (a1)--(a2)--(a3)--(a4)--(a5)--cycle;
			\foreach \x in {1,2,3,...,5}{
				\filldraw [black] (a\x) circle (1.5pt);
			}
			\node[right,yshift=0.34cm,xshift=0cm,yscale=0.7,xscale=0.7] at (A) {$\mathcal{B}_{\mathcal{K}^{\langle 1\rangle}}$};
		%drugi
			\path (A) ++(-126:1cm)++(0,-1) coordinate (B);
			\foreach \x in {1,2,3,...,5}{
				\path (B) ++(\x*72-54:1cm) coordinate (b\x);
			}
			\filldraw[fill=lightgray,draw=black] (b1)--(b2)--(b3)--(b4)--(b5)--cycle;
			\foreach \x in {1,2,3,...,5}{
				\filldraw [black] (b\x) circle (1.5pt);
			}
			\draw[->] (a3) + (-36:0.5cm) arc (-36:108:0.5cm) node[pos=0.5,right,yshift=-5mm,xshift=1mm, yscale=0.8, xscale=0.8]{$ \frac{4}{5} \pi$};

			\foreach \x in {2,3,...,10}{
				\path (A) ++(\x*36+18:3cm) coordinate (\x);
				\draw[dashed] (A)--(\x);
			}
			\path (A) ++(1*36+18:0.9cm) coordinate (1);
			\path (1) ++(1*36+18:2.1cm) coordinate (11);
			\draw[dashed] (1)--(11);
			\filldraw [black] (A) circle (2.5pt);
			
			%%%%%%
			\coordinate (J) at (6.5,0);
			\foreach \x in {1,2,3,...,5}{
				\path (J) ++(\x*72+18:1cm) coordinate (j\x);
			}
			\filldraw[fill=lightgray,draw=black] (j1)--(j2)--(j3)--(j4)--(j5)--cycle;
			\foreach \x in {1,2,3,...,5}{
				\filldraw [black] (j\x) circle (1.5pt);
			}
			\node[right,yshift=0.34cm,xshift=0cm,yscale=0.7,xscale=0.7] at (J) {$\mathcal{B}_{\mathcal{K}^{\langle 1\rangle}}$};
			%drugi
			\path (J) ++(-126:1cm)++(-162:1cm) coordinate (K);
			\foreach \x in {1,2,3,...,5}{
				\path (K) ++(\x*72+18:1cm) coordinate (k\x);
			}
			\filldraw[fill=lightgray,draw=black] (k1)--(k2)--(k3)--(k4)--(k5)--cycle;
			\foreach \x in {1,2,3,...,5}{
				\filldraw [black] (k\x) circle (1.5pt);
			}
			\draw[->] (k5) + (-108:0.5cm) arc (-108:108:0.5cm) node[pos=0.5,right,yshift=-3.5mm,xshift=-0.2mm]{$ \frac{6}{5} \pi$};

			\foreach \x in {2,3,...,10}{
				\path (J) ++(\x*36+18:3cm) coordinate (L\x);
				\draw[dashed] (J)--(L\x);
			}
			\path (J) ++(1*36+18:0.9cm) coordinate (tu);
			\path (tu) ++(1*36+18:2.1cm) coordinate (tam);
			\draw[dashed] (tu)--(tam);
			\filldraw [black] (J) circle (2.5pt);
			
			%%%%%%
			\coordinate (D) at (13,0);
			\foreach \x in {1,2,3,...,5}{
				\path (D) ++(\x*72+18:1cm) coordinate (d\x);
			}
			\filldraw[fill=lightgray,draw=black] (d1)--(d2)--(d3)--(d4)--(d5)--cycle;
			\foreach \x in {1,2,3,...,5}{
				\filldraw [black] (d\x) circle (1.5pt);
			}
			\node[right,yshift=0.34cm,xshift=0cm,yscale=0.7,xscale=0.7] at (D) {$\mathcal{B}_{\mathcal{K}^{\langle 1\rangle}}$};
			%drugi
			\path (D) ++(-126:2cm) coordinate (C);
			\foreach \x in {1,2,3,...,5}{
				\path (C) ++(\x*72+54:1cm) coordinate (c\x);
			}
			\filldraw[fill=lightgray,draw=black] (c1)--(c2)--(c3)--(c4)--(c5)--cycle;
			\foreach \x in {1,2,3,...,5}{
				\filldraw [black] (c\x) circle (1.5pt);
			}
			\draw[->] (c5) + (-72:0.5cm) arc (-72:108:0.5cm) node[pos=0.5,right,yshift=-3.5mm,xshift=0.4mm]{$\pi$};

			\foreach \x in {2,3,...,10}{
				\path (D) ++(\x*36+18:3cm) coordinate (A\x);
				\draw[dashed] (D)--(A\x);
			}
			\path (D) ++(1*36+18:0.9cm) coordinate (1111);
			\path (1111) ++(1*36+18:2.1cm) coordinate (111);
			\draw[dashed] (1111)--(111);
			\filldraw [black] (D) circle (2.5pt);
		\end{tikzpicture}
		\caption{Example for $k=5$. From the left side, two possible rotations and the impossible one forced by the symmetry.}
		\label{5i5}
		\end{figure}
		
		If $k>6$ is even, we will show that even though any 0-complex $\Delta_{0}$ adjacent to $\Delta_{\mathcal {B}} $ does not violate the symmetry with respect to the axis passing through the common vertex of $\Delta_{\mathcal {B}} $ and $\Delta_{0}$, at least two other symmetry axes of $\mathcal{K}^{\langle 1 \rangle}$ intersect $\Delta_0$, what contradicts the symmetry condition.
		
		We know that the internal angle of convex hulls $\mathcal{H}(\Delta_{\mathcal{B}})$ and $\mathcal{H}(\Delta_{0})$ is $\frac{k-2}{k} \pi$. They are regular polygons, and let us denote the length of their sides as $2a$. 
		
		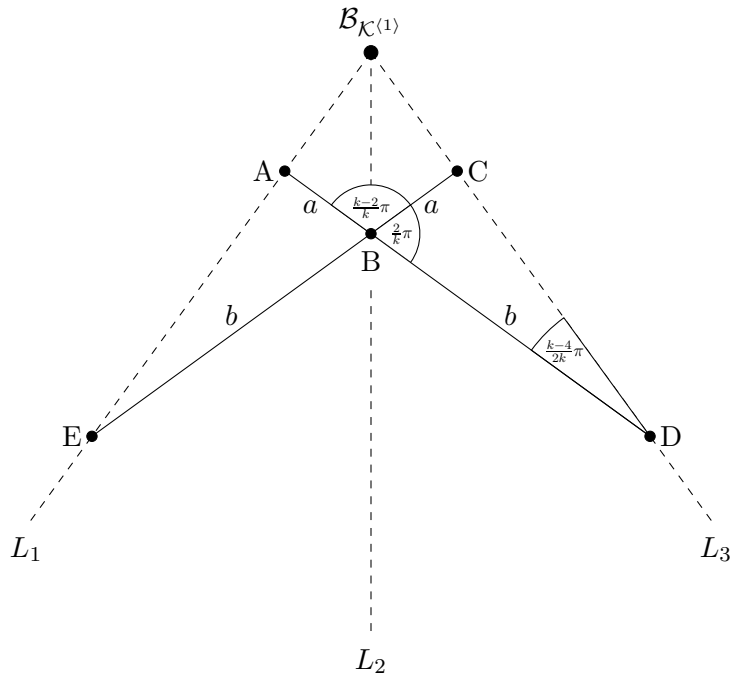
\begin{figure}[H]
		\centering
		\begin{tikzpicture}[scale=1.3]
			\coordinate (0) at (0,0) node[above, yshift=0.08cm]{$\mathcal{B}_{\mathcal{K}^{\langle 1\rangle}}$};
			\filldraw [black] (0) circle (2pt);
			\foreach \x in {1,2,3}{
				\path (0) ++(-162+36*\x:6cm) coordinate (L\x );
				\draw[dashed] (0) -- (L\x) node[below]{$ L_\x $};
			}
			\path (0,0)++(-162+36:1.5cm) node[left]{A} coordinate (A);
			\path (0,0)++(-54:1.5cm) node[right]{C} coordinate(C);
			\path(A) ++(-36:1.089813792008) node[below]{D} coordinate (D);
			\draw(C)--(D) node[below, pos=0.3]{$ a $};
			\draw(A) -- (D) node[below, pos=0.3]{$ a $};
			\path(A) ++(-36:4.616525305762879) node[right]{D} coordinate (B);
			\path(C) ++(180+36:4.616525305762879) node[left]{E} coordinate (E);
			
			\path(D)++(36:0.5cm) coordinate (D');
			\draw[fill=white] (D) --(D') arc (36:144:0.5cm) node[pos=0.5,xscale=0.6,yscale=0.6,yshift=-0.5cm]{$ \frac{k-2}{k} \pi$} --cycle;
			\path(D)++(36:0.5cm) coordinate (D'');
			\draw[fill=white] (D) --(D'') arc (36:-36:0.5cm) node[pos=0.5,xscale=0.6,yscale=0.6,xshift=-0.4cm]{$ \frac{2}{k} \pi$} --cycle;
			\path(D)++(-36:0.5cm) coordinate (D''');
			\filldraw[fill=white,draw=white] (D) --(D''') arc (-36:-144:0.5cm)  --cycle;
			\path(B)++(126:1.5cm) coordinate (B');
			\path(A) ++(-36:1.089813792008) node[below,yshift=-0.1cm]{B} coordinate (D);
			\draw[fill=white] (B) --(B') arc (126:144:1.5cm) node[pos=0.5,xscale=0.6,yscale=0.6,yshift=-0.4cm,xshift=0.4cm]{$ \frac{k-4}{2k} \pi$} --cycle;
			\draw(E)--(D) node[above, pos=0.5]{$ b $};
			\draw(D) -- (B) node[above, pos=0.5]{$ b $};
			
			\filldraw [black] (A) circle (1.5pt);
			\filldraw [black] (B) circle (1.5pt);
			\filldraw [black] (C) circle (1.5pt);
			\filldraw [black] (D) circle (1.5pt);
			\filldraw [black] (E) circle (1.5pt);
		\end{tikzpicture}
		\caption{We get $b<2a$, so the vertices of $\Delta_0$ lie outside the axes $L_1$ and $ L_3$.}\label{Styczne d0}
		\end{figure}
		Let us denote the vertices as shown in the Figure \ref{Styczne d0}. Let $B$ be the common vertex of $\mathcal{H}(\Delta_{\mathcal {B}})$ and $\mathcal{H}(\Delta_{0})$. Let us also denote the axis of symmetry of $\mathcal{K}^{\langle 1 \rangle} $ passing through $B$ as $L_2$. Let $A$ and $C$ be the midpoints of the sides of $\mathcal{H} (\Delta_{\mathcal {B}})$ sides emanating from $B$, and let the axes of symmetry passing through them be $L_1, L_3$ respectively. Then $AB = BC = a$.
				
		Let $D$ and $E$ be the intersection points of $L_3$ and $L_1$ with the sides of $\mathcal{H} (\Delta_0)$ emanating from $B$. Let us denote the length of $BD$ as $b$. We will prove that $b<2a $, and hence the vertex of $\Delta_{0}$ lying on the extension of $AD$ is not in the area between the axes denoted as $L_2$ and $L_3$.
		
		After simple calculations, we get $\angle BDC = \frac{k-4}{2k} \pi $. From the dependences in a right triangle $BCD$ we get
		$$
		 b=\frac{a}{\sin\left( \frac{k-4}{2k}\pi\right)}.
		$$ 
		We want to prove that $b<2a $. We make further transformations of our inequality
		\begin{eqnarray}
		\frac{a}{\sin\left( \frac{k-4}{2k}\pi\right) }&<&2a\\
		\frac{1}{\sin\left( \frac{k-4}{2k}\pi\right) }&<&2\\
		1&<&2\sin\left( \frac{k-4}{2k}\pi\right) .
		\end{eqnarray}
		It's easy to see that $ \sin \left( \frac {k-4} {2k} \pi \right) $ is increasing for $k\geq 4 $.
		Hence for $ k \geq $ 8 we have $ 2 \sin \left(\frac {7-4} {2 \cdot 7} \pi \right) <2 \sin \left(\frac {k-4} {2 \cdot k} \pi \right) $, and since $ 1 <1.24 <2 \sin \left(\frac {3} {14} \pi \right) $, we have proved the lemma.
	\end{proof}
	After presentation of all the necessary tools, we can finally move on to the actual proof of the Theorem \ref{tw:niep}.
	\begin{proof}[Proof of Theorem \ref{tw:niep}]
	
		$ (\Rightarrow )$ If $\mathcal{K}^{\langle 1\rangle}$ can be well labeled, then any subset of $\mathcal{K}^{\langle 1\rangle}$ can also be well labeled.\\
	
		$ (\Leftarrow)$  If $k\in \{3,4,5\}$, then Theorem \ref{thm:prime} and \cite{bib:KOPP}[Corollary 3.1] tell that the fractal has the GLP. Let us assume that $k\geq 6$ (but if $k=6$, then $\mathcal{B}_{\mathcal{K}^{\langle 1\rangle}} \notin \mathcal{K}^{\langle 1\rangle} $).
		
		Let $\mathcal{A}=\{1,2,3,\dots,k\}$ be the set of labels. From Lemma \ref{lem:niezmienniczosc na obroty} we know that the areas $U_i$ are $2\pi/k$-rotationally invariant, therefore $W_i$'s are as well. Let us denote one of them as $W_1$ and the following as $W_2, ..., W_k$ counter-clockwise.
		Let us label vertices of all $0$-complexes in $ W_1\cup W_2 $ in a "good" way, i.e. each $0$-complex has the complete set of labels in the same order. We will expand this labeling to vertices of all other $0$-complexes in $\mathcal{K}^{\langle 1\rangle}$.

Let us consider vertices of $0$-complexes $\Delta_0$ from $W_1$ being simultaneoulsy vertices of another complex $\Delta_0^{'}$ from
		\begin{itemize}
			\item[1.]  $ W_k $ -- we denote them as $ w_{1,i} $, $ i\in\mathbb{I}_1 $ and their collection as $V_1$.
			\item[2.]  $ W_2 $ -- we denote them as $ w_{2,i} $, $ i\in\mathbb{I}_2 $, and their collection as $V_2$. 
		\end{itemize} 
		From Lemma \ref{a} we know that no vertex belongs do both categories. Lemma \ref{D0 w BK1} implies that each complex belongs to some slice.

		From Lemma \ref{lem:niezmienniczosc na obroty} we obtain that for each $w_{1,i} \in V_1$ there exists exactly one $w_{2,i}\in V_2$ such that after a rotation by angle $\frac{2\pi}{k}$ (counter clockwise) it will coincide with $w_{2,i} $.
		In particular, the sets  $\mathbb{I}_1 $ i $\mathbb{I}_2 $ have the same cardinality. Number of their elements will be denoted as $n$.\\
		We will now pair the points $w_{1,i}$ and $w_{2,i}$ which coincide after rotation. We can renumerate them in such way that in each pair the second indexes are the same $(w_{1,1},w_{2,1}),\dots,(w_{1,n},w_{2,n})$. For each pair $ (w_{1,i},w_{2,i})$ we introduce the difference of labels
		$$
		r_i = \ell(w_{1,i})-\ell(w_{2,i}).
		$$
		We will show that
		$$
		r_1\equiv r_2\equiv \dots \equiv r_n\quad \mod k
		$$ 
		and that it is sufficient to prove that the whole $\mathcal{K}^{\langle 1\rangle}$ can be well labeled.
					
		Firstly, we show that it is sufficient.
		
		Since labels in pairs differ by $r \mod k$, then the vertices of complexes in $W_2$ have labels increased by $r \mod k$ from the labels of respective vertices from $W_1$. Analogoulsly, the labels of vertices in $W_3$ are increased by $2r \mod k$ from respective vertices from $W_1$, ... and in $W_k$ are increased by $(k-1)r \mod k$.
		By continuing this operation on complexes in $W_1$ we would not change labels of its vertices, because they would be increased by $kr \equiv 0 \mod k$. Therefore each vertex is given exactly one label and the orientation of labels on each complex is preserved. This means that the labeling is good.

		Secondly, we show that $ r_i\equiv  r_j \mod k$ for each $i$ and $j$.
		
		If all differences $r_i$ are equivalent $\mod k$, then the proof is over. Therefore let us assume that there exist at least two differences.
		
		We pick pairs $ (w_{1,1}, w_{2,1}), (w_{1,2},w_{2,2})$ of vertices such that the differences of labels $r_1$ and $r_2$ are distinct. We introduce $r_{21}:= \ell(w_{2,1}) - \ell(w_{1,2})$.
		
		The points $w_{2,1}, w_{2,2}$ are vertices of $0$-complexes from $W_2$. After rotation by angle $\frac{2\pi}{k}$ (counter clockwise) $ w_{2,1}$ and $w_{2,2} $ will coincide with vertices which will be denoted as $w_{3,1}$ and $ w_{3,2}$. Analogoulsy $w_{1,1}$ and $w_{1,2}$ will coincide after rotation with $ w_{2,1}$ and $w_{2,2}$, as it can be seen at Figure \ref{fig:wycinek}.

	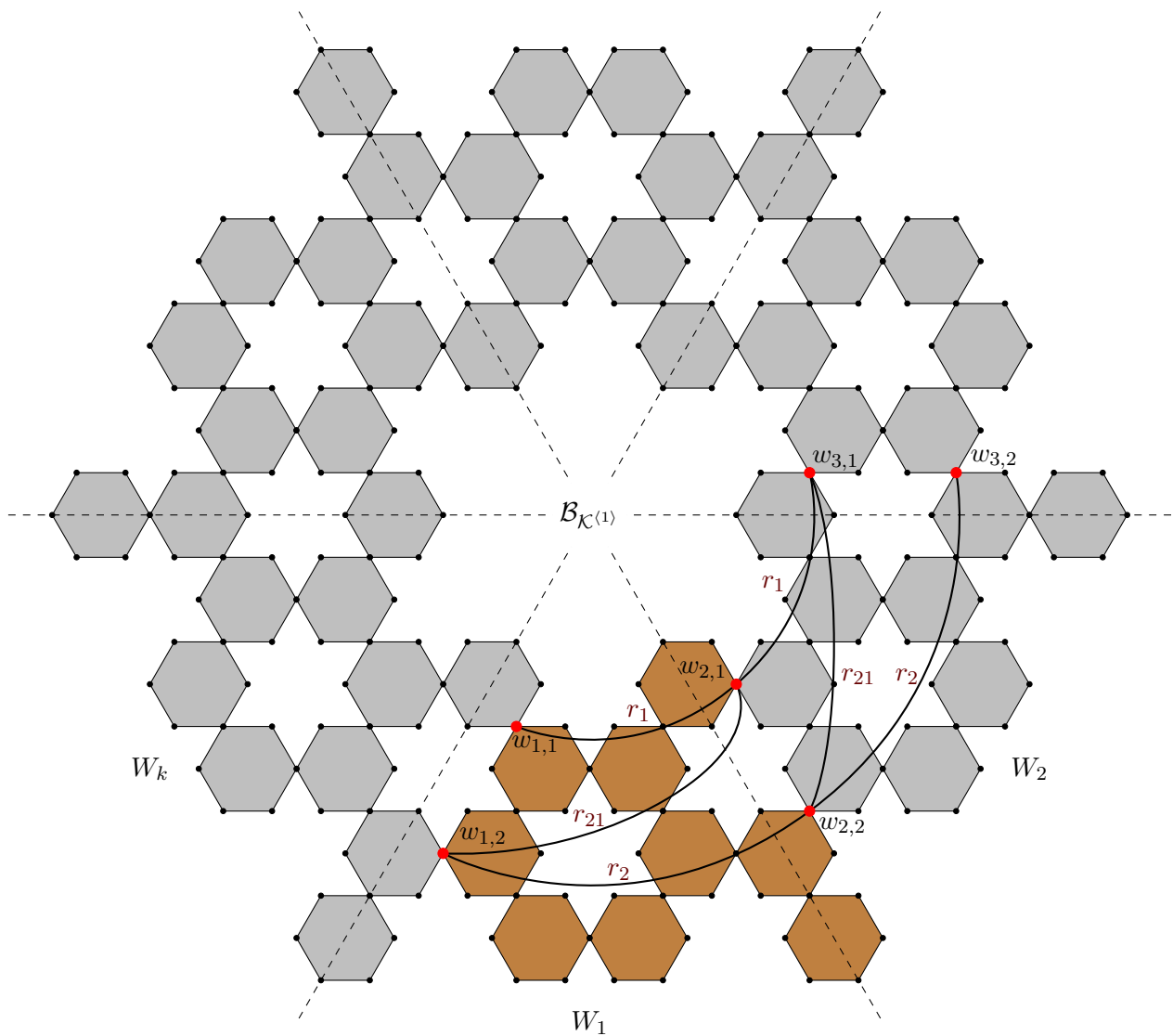
\begin{figure}[H]
			\begin{center}
			\begin{tikzpicture}[scale=0.7]
				%1/6, 
				\sz{(\kat{0}:2)}
				\sz{$(\kat{0}:2)+(\kat{1}:2)$}
				\szb{$2*(\kat{0}:2)+(\kat{1}:2)$}
				\szb{$2*(\kat{0}:2)+2*(\kat{1}:2)$}
				\sz{$2*(\kat{0}:2)+2*(\kat{1}:2)+(\kat{2}:2)$}
				\szb{$3*(\kat{0}:2)+2*(\kat{1}:2)$}
				\szb{$3*(\kat{0}:2)+3*(\kat{1}:2)$}
				\szb{$2*(\kat{0}:2)+(\kat{1}:2)+(\kat{5}:2)$}
				\szb{$3*(\kat{0}:2)+(\kat{1}:2)+(\kat{5}:2)$}
				\szb{$3*(\kat{0}:2)+2*(\kat{1}:2)+(\kat{5}:2)$}
				\szb{$4*(\kat{0}:2)+2*(\kat{1}:2)+(\kat{5}:2)$}
				\szb{$4*(\kat{0}:2)+2*(\kat{1}:2)+2*(\kat{5}:2)$}
				
				%1/6,
				\sz{$4*(\kat{0}:2)+3*(\kat{1}:2)+(\kat{5}:2)$}
				\sz{$4*(\kat{0}:2)+3*(\kat{1}:2)+(\kat{2}:2)+(\kat{5}:2)$}
				\sz{$4*(\kat{0}:2)+4*(\kat{1}:2)+(\kat{2}:2)+(\kat{5}:2)$}
				\sz{$4*(\kat{0}:2)+4*(\kat{1}:2)+2*(\kat{2}:2)+(\kat{5}:2)$}
				\sz{$5*(\kat{0}:2)+4*(\kat{1}:2)+(\kat{2}:2)+(\kat{5}:2)$}
				\sz{$5*(\kat{0}:2)+4*(\kat{1}:2)+(\kat{2}:2)+2*(\kat{5}:2)$}
				\sz{$4*(\kat{0}:2)+6*(\kat{1}:2)+(\kat{2}:2)+2*(\kat{5}:2)$}
				\sz{$5*(\kat{0}:2)+6*(\kat{1}:2)+(\kat{2}:2)+2*(\kat{5}:2)$}
				\sz{$5*(\kat{0}:2)+3*(\kat{1}:2)+(\kat{5}:2)$}
				
				%2/6
				\sz{$4*(\kat{0}:2)+6*(\kat{1}:2)+2*(\kat{2}:2)+2*(\kat{5}:2)$}
				\sz{$3*(\kat{0}:2)+6*(\kat{1}:2)+2*(\kat{2}:2)+2*(\kat{5}:2)$}
				\sz{$3*(\kat{0}:2)+6*(\kat{1}:2)+3*(\kat{2}:2)+2*(\kat{5}:2)$}
				\sz{$3*(\kat{0}:2)+7*(\kat{1}:2)+3*(\kat{2}:2)+2*(\kat{5}:2)$}
				\sz{$4*(\kat{0}:2)+7*(\kat{1}:2)+3*(\kat{2}:2)+2*(\kat{5}:2)$}
				\sz{$4*(\kat{0}:2)+7*(\kat{1}:2)+3*(\kat{2}:2)+3*(\kat{5}:2)$}
				\sz{$2*(\kat{0}:2)+6*(\kat{1}:2)+3*(\kat{2}:2)+2*(\kat{5}:2)$}
				\sz{$2*(\kat{0}:2)+8*(\kat{1}:2)+3*(\kat{2}:2)+2*(\kat{5}:2)$}
				\sz{$2*(\kat{0}:2)+9*(\kat{1}:2)+3*(\kat{2}:2)+2*(\kat{5}:2)$}
				
				%3/6
				\sz{$(\kat{0}:2)+8*(\kat{1}:2)+3*(\kat{2}:2)+2*(\kat{5}:2)$}
				\sz{$(\kat{0}:2)+7*(\kat{1}:2)+3*(\kat{2}:2)+2*(\kat{5}:2)$}
				\sz{$7*(\kat{1}:2)+3*(\kat{2}:2)+2*(\kat{5}:2)$}
				\sz{$7*(\kat{1}:2)+3*(\kat{2}:2)+(\kat{5}:2)$}
				\sz{$6*(\kat{1}:2)+3*(\kat{2}:2)+2*(\kat{5}:2)$}
				\sz{$8*(\kat{1}:2)+3*(\kat{2}:2)+(\kat{5}:2)$}
				\sz{$(\kat{0}:2)+8*(\kat{1}:2)+3*(\kat{2}:2)+(\kat{5}:2)$}
				\sz{$7*(\kat{1}:2)+3*(\kat{2}:2)+(\kat{3}:2)+(\kat{5}:2)$}
				\sz{$7*(\kat{1}:2)+4*(\kat{2}:2)+(\kat{3}:2)+(\kat{5}:2)$}
				\sz{$7*(\kat{1}:2)+3*(\kat{2}:2)+(\kat{3}:2)+(\kat{4}:2)+(\kat{5}:2)$}
				
				%4/6
				\sz{$7*(\kat{1}:2)+3*(\kat{2}:2)+(\kat{3}:2)+(\kat{4}:2)+2*(\kat{5}:2)$}
				\sz{$7*(\kat{1}:2)+3*(\kat{2}:2)+2*(\kat{3}:2)+(\kat{4}:2)+(\kat{5}:2)$}
				\sz{$6*(\kat{1}:2)+3*(\kat{2}:2)+2*(\kat{3}:2)+(\kat{4}:2)+(\kat{5}:2)$}	
				\sz{$6*(\kat{1}:2)+2*(\kat{2}:2)+2*(\kat{3}:2)+(\kat{4}:2)+(\kat{5}:2)$}				 
				\sz{$6*(\kat{1}:2)+2*(\kat{2}:2)+(\kat{3}:2)+(\kat{4}:2)+(\kat{5}:2)$}				 
				\sz{$6*(\kat{1}:2)+(\kat{2}:2)+(\kat{3}:2)+(\kat{4}:2)+(\kat{5}:2)$}			
				\sz{$6*(\kat{1}:2)+(\kat{2}:2)+4*(\kat{3}:2)+(\kat{4}:2)+(\kat{5}:2)$}	
				\sz{$6*(\kat{1}:2)+(\kat{2}:2)+3*(\kat{3}:2)+(\kat{4}:2)+(\kat{5}:2)$}				 
				
				%5/6
				\sz{$6*(\kat{1}:2)+(\kat{2}:2)+3*(\kat{3}:2)+(\kat{4}:2)+2*(\kat{5}:2)$}	
				\sz{$6*(\kat{1}:2)+(\kat{2}:2)+2*(\kat{3}:2)+(\kat{4}:2)+2*(\kat{5}:2)$}			
				\sz{$6*(\kat{1}:2)+(\kat{2}:2)+2*(\kat{3}:2)+(\kat{4}:2)+3*(\kat{5}:2)$}				 
				\sz{$5*(\kat{1}:2)+(\kat{2}:2)+2*(\kat{3}:2)+(\kat{4}:2)+3*(\kat{5}:2)$}	
				\sz{$5*(\kat{1}:2)+(\kat{2}:2)+3*(\kat{3}:2)+(\kat{4}:2)+3*(\kat{5}:2)$}	
				\sz{$5*(\kat{1}:2)+(\kat{2}:2)+3*(\kat{3}:2)+(\kat{4}:2)+2*(\kat{5}:2)$}

				\coordinate (C) at (0,0);	 
				\path (C) ++(60:10cm) ++(0:2cm)coordinate (B);
				
				\foreach \x in {0,1,2,...,5}{
					\path (B) ++(\x *60:12cm) coordinate (B\x);
					\draw[dashed] (B)--(B\x);
				}
				\filldraw[white] (B) circle (21.5pt);
				\node at (B) {$\mathcal{B}_{\mathcal{K}^{\langle 1\rangle}}$};
				\path (B3)++(-60:6) node{$ W_k $};
				\path (B4)++(0:6) node{$ W_1 $};
				\path (B5)++(60:6) node{$ W_2 $};
				
				\path (C) ++(0:3cm)++(60:2) coordinate (w12) node[above right, xshift=1mm, yshift=-0.5mm]{$ w_{1,2} $};
				
				\path (w12)++(60:3) coordinate (w11)node[below right, xshift=-2mm]{$ w_{1,1} $};
				
				\path (w11)++(0:4)++(60:1) coordinate (w21)node[above left, yshift=-0.9mm]{$ w_{2,1} $};			
				
				\path (w21)++(-60:3) coordinate (w22)node[below right, yshift=0.5mm]{$ w_{2,2} $};
				
				\path (w21)++(60:4)++(120:1) coordinate (w31)node[above right, yshift=-1mm, xshift=-1mm]{$ w_{3,1} $};
				
				\path (w31)++(0:3) coordinate (w32)node[above right, yshift=-0.8mm, xshift=0.7mm]{$ w_{3,2} $};
				
				\draw[thick] plot [smooth,tension=1.3] coordinates{(w11)(w21)(w31)};
				\node[above,color=black!60!red] at (8,4.25){$ r_1 $};
				\node[left,color=black!60!red] at (11.2,7.25){$ r_1 $};
				\draw[thick] plot [smooth,tension=1.3] coordinates{(w12)(8.5,2.8)(w21)};
				\node[above,color=black!60!red] at (7.6,1){$ r_2 $};
				\node[above,color=black!60!red] at (13.5,5){$ r_2 $};
				\draw[thick] plot [smooth,tension=1.3] coordinates{(w12)(w22)(w32)};
				\draw[thick] plot [smooth,tension=1.3] coordinates{(w22)(12,6)(w31)};
				\node[above,color=black!60!red] at (12.5,5){$ r_{21} $};
				\node[above,color=black!60!red] at (7,2.1){$ r_{21} $};
				\filldraw [red] (w12) circle (3pt);
				\filldraw [red] (w11) circle (3pt);
				\filldraw [red] (w21) circle (3pt);
				\filldraw [red] (w22) circle (3pt);
				\filldraw [red] (w31) circle (3pt);
				\filldraw [red] (w32) circle (3pt);
			\end{tikzpicture}
			\label{fig:wycinek}
		\end{center}
		\caption{Slices $W_1$, $W_2$ and $W_k$ in $\mathcal{K}^{\langle 1\rangle}$ for $k=6$. The slice $W_1$ is colored brown. Vertices bordering slices are marked red. Lines connecting red vertices are denoted with differences $r_i$ between labels.}	
	\end{figure}
	
	We assumed that vertices of all complexes $\Delta_0 \subset W_1 \cup W_2$ are well labeled, so in particular $w_{1,1}, w_{1,2}, w_{2,1}, w_{2,2}, w_{3,1}, w_{3,2}$ have assigned labels.
	
	Due to rotation invariance $ \ell(w_{3,1}) - \ell(w_{2,1}) \equiv r_1 \mod k$, as well as $\ell(w_{3,2})-\ell(w_{2,2}) \equiv r_2 \mod k$ and $\ell(w_{3,1}) - \ell(w_{2,2}) \equiv r_{21}\mod k.$ 
				
		Therefore
	\begin{align*}
		r_1 + r_{21} = &(\ell(w_{3,1}) - \ell(w_{2,1})) + (\ell(w_{2,1}) - \ell(w_{1,2}))  = \ell(w_{3,1}) - \ell(w_{1,2}) = \\
		&(\ell(w_{3,1}) - \ell(w_{2,2})) + (\ell(w_{2,2}) - \ell(w_{1,2})) \equiv r_{21} + r_2 \mod k
		\end{align*}
		
		What gives
		$$
		r_1\equiv r_2\mod k
		$$
		contrary to our assumption.
	\end{proof}
	The natural question arises. Why did we need to assume good labeling of "two $k$-th" of the fractal, not only "one $k$-th"?
	The answer for odd $k$ is given by Example 3.5 in \cite{bib:KOPP}, where one can label "one $k$-th" of $\mathcal{K}^{\langle 1\rangle}$, but the fractal does not have GLP.
	
	In case of even $k$ further reduction is possible.
	\begin{definition}
The closed slice $\overline{W_i}$ is the sum of $0$-complexes $\Delta_0$, which centers lie in the closure of area $U_i$, but still without the barycenter $\mathcal{K}^{\langle 1\rangle}$.
	\end{definition}
	Naturally, in this case intersections of neighbor $\overline{W_j}$'s are nonempty.

	The necessity of considering closed slices is shown at Figure \ref{W1dla10}. We can see that good labeling of $W_1$ is possible, because we can divide complexes $ \Delta_{0} \subset W_1$ into two classess. The closed slice $\overline{W_1}$ consisting of additional brown complexes cannot be well labeled and the whole fractal does not have GLP.

	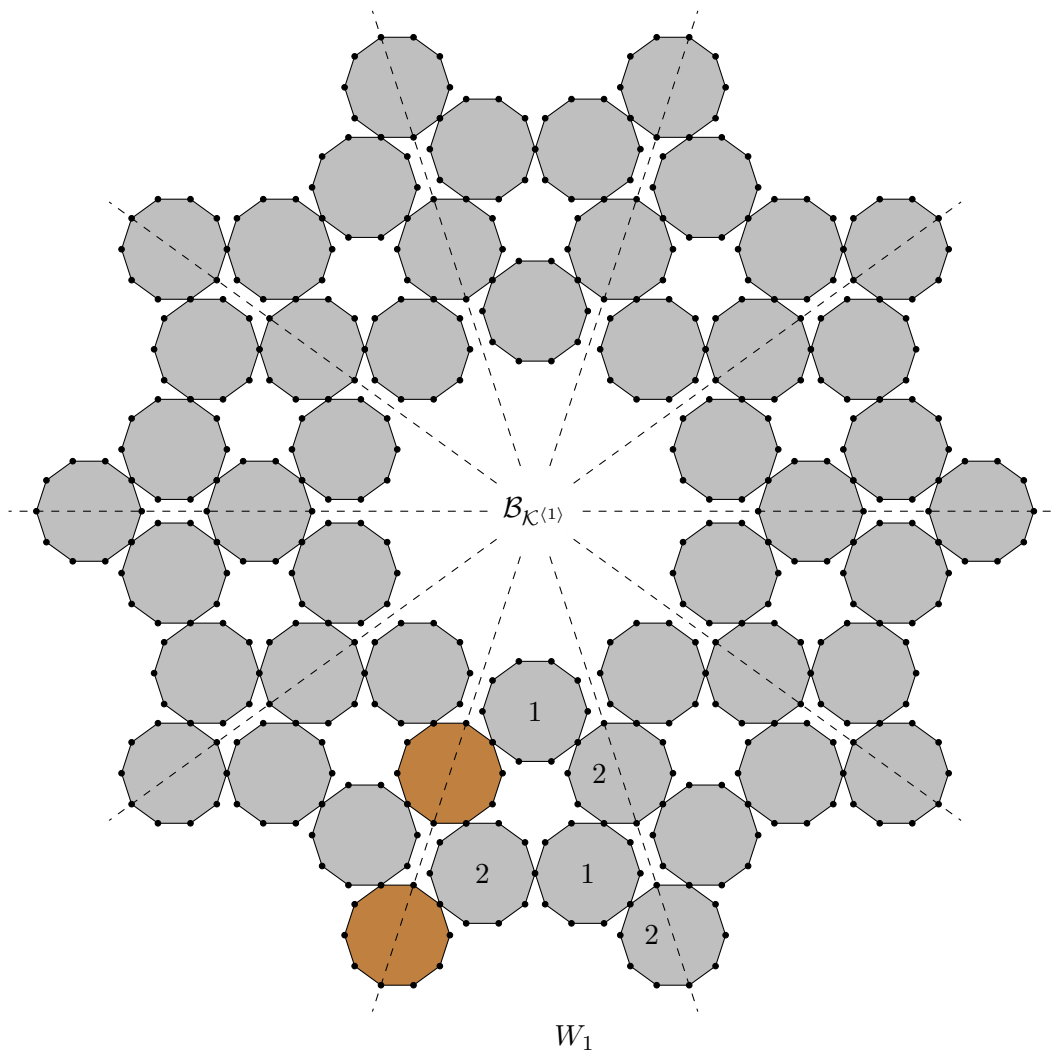
\begin{figure}[H]
		\centering
		\begin{tikzpicture}[scale=0.7]
			%1/10
			\ds{(\kt{0}:0)}
			\node at($(\kt{0}:0)$) {2};
			\dsone{$(\kt{0}:2)$}
			\dstwo{$(\kt{0}:2)+(\kt{2}:2)$}
			\dsone{$(\kt{0}:2)+(\kt{2}:2)+(\kt{4}:2)$}
			\dsoneb{$(\kt{3}:2)$}
			\dsoneb{$(\kt{6}:2)$}%lewy dolny
			\dstwo{$(\kt{0}:2)+(\kt{9}:2)$}%prawy dolny
			
			%2/10
			\ds{$(\kt{0}:2)+(\kt{2}:2)+(\kt{9}:2)$}
			\ds{$(\kt{0}:2)+(\kt{1}:2)+(\kt{2}:2)+(\kt{9}:2)$}%najbardziej  na prawo
			\ds{$(\kt{0}:2)+(\kt{1}:2)+(\kt{2}:2)+(\kt{3}:2)+(\kt{9}:2)$}
			\ds{$(\kt{1}:2)+(\kt{2}:2)+(\kt{3}:2)+(\kt{9}:2)$}
			\ds{$2*(\kt{0}:2)+(\kt{1}:2)+(\kt{2}:2)+(\kt{9}:2)$}
			
			%3/10
			\ds{$2*(\kt{0}:2)+(\kt{1}:2)+(\kt{2}:2)+(\kt{3}:2)+(\kt{9}:2)$}
			\ds{$2*(\kt{0}:2)+(\kt{1}:2)+2*(\kt{2}:2)+(\kt{3}:2)+(\kt{9}:2)$}
			\ds{$2*(\kt{0}:2)+(\kt{1}:2)+2*(\kt{2}:2)+(\kt{3}:2)$}
			\ds{$2*(\kt{0}:2)+2*(\kt{2}:2)+(\kt{3}:2)$}
			\ds{$2*(\kt{0}:2)+2*(\kt{1}:2)+2*(\kt{2}:2)+(\kt{3}:2)+(\kt{9}:2)$}
			
			%4/10
			\ds{$2*(\kt{0}:2)+2*(\kt{1}:2)+2*(\kt{2}:2)+(\kt{3}:2)$}
			\ds{$2*(\kt{0}:2)+2*(\kt{1}:2)+2*(\kt{2}:2)+2*(\kt{3}:2)$}
			\ds{$2*(\kt{0}:2)+2*(\kt{1}:2)+3*(\kt{2}:2)+2*(\kt{3}:2)$}
			\ds{$(\kt{0}:2)+2*(\kt{1}:2)+2*(\kt{2}:2)+2*(\kt{3}:2)$}
			\ds{$(\kt{0}:2)+2*(\kt{1}:2)+(\kt{2}:2)+2*(\kt{3}:2)$}
			
			%5/10
			\ds{$(\kt{0}:2)+2*(\kt{1}:2)+3*(\kt{2}:2)+2*(\kt{3}:2)$}
			\ds{$2*(\kt{1}:2)+2*(\kt{2}:2)+2*(\kt{3}:2)$}
			\ds{$(\kt{0}:2)+2*(\kt{1}:2)+3*(\kt{2}:2)+2*(\kt{3}:2)+(\kt{4}:2)$}
			\ds{$2*(\kt{1}:2)+2*(\kt{2}:2)+3*(\kt{3}:2)$}
			\ds{$(\kt{0}:2)+2*(\kt{1}:2)+3*(\kt{2}:2)+3*(\kt{3}:2)+(\kt{4}:2)$}
			
			%6/10
			\ds{$2*(\kt{1}:2)+2*(\kt{2}:2)+4*(\kt{3}:2)$}
			\ds{$2*(\kt{1}:2)+2*(\kt{2}:2)+4*(\kt{3}:2)+(\kt{5}:2)$}
			\ds{$2*(\kt{1}:2)+(\kt{2}:2)+4*(\kt{3}:2)+(\kt{5}:2)$}
			\ds{$2*(\kt{1}:2)+2*(\kt{2}:2)+4*(\kt{3}:2)+(\kt{4}:2)+(\kt{5}:2)$}
			\ds{$2*(\kt{1}:2)+(\kt{2}:2)+4*(\kt{3}:2)+(\kt{5}:2)+(\kt{9}:2)$}
			
			%7/10
			\ds{$2*(\kt{1}:2)+4*(\kt{3}:2)+(\kt{5}:2)$}
			\ds{$2*(\kt{1}:2)+(\kt{2}:2)+4*(\kt{3}:2)+(\kt{4}:2)+(\kt{5}:2)$}
			\ds{$(\kt{1}:2)+(\kt{2}:2)+4*(\kt{3}:2)+(\kt{4}:2)+(\kt{5}:2)$}
			\ds{$(\kt{1}:2)+(\kt{2}:2)+4*(\kt{3}:2)+(\kt{4}:2)+2*(\kt{5}:2)$}
			\ds{$2*(\kt{1}:2)+4*(\kt{3}:2)+2*(\kt{5}:2)$}
			
			%8/10
			\ds{$2*(\kt{1}:2)+4*(\kt{3}:2)+3*(\kt{5}:2)$}
			\ds{$2*(\kt{1}:2)+4*(\kt{3}:2)+3*(\kt{5}:2)+(\kt{7}:2)$}
			\ds{$2*(\kt{1}:2)+3*(\kt{3}:2)+2*(\kt{5}:2)$}
			\ds{$(\kt{1}:2)+4*(\kt{3}:2)+3*(\kt{5}:2)+(\kt{7}:2)$}
			\ds{$(\kt{1}:2)+3*(\kt{3}:2)+2*(\kt{5}:2)$}
			
			%9/10
			\ds{$3*(\kt{3}:2)+2*(\kt{5}:2)$}
			\ds{$(\kt{1}:2)+3*(\kt{3}:2)+2*(\kt{5}:2)+(\kt{9}:2)$}
			\ds{$2*(\kt{3}:2)+2*(\kt{5}:2)$}
			\ds{$(\kt{1}:2)+3*(\kt{3}:2)+2*(\kt{5}:2)+(\kt{7}:2)+(\kt{9}:2)$}
			\ds{$(\kt{1}:2)+3*(\kt{3}:2)+3*(\kt{5}:2)+2*(\kt{7}:2)+(\kt{9}:2)$}
			
			%10/10
			\ds{$(\kt{1}:2)+3*(\kt{3}:2)+(\kt{5}:2)+(\kt{7}:2)+(\kt{9}:2)$}
			\ds{$(\kt{1}:2)+3*(\kt{3}:2)+2*(\kt{5}:2)+2*(\kt{7}:2)+(\kt{9}:2)$}
			\ds{$(\kt{1}:2)+3*(\kt{3}:2)+2*(\kt{5}:2)+2*(\kt{7}:2)+2*(\kt{9}:2)$}
			
			%srodek
			\coordinate (bar) at ($(\kt{0}:2)+2*(\kt{2}:2)+(\kt{3}:2)+(\kt{4}:2)$);	 
			\foreach \x in {0,1,2,...,9}{
				\path (bar) ++(\x *36:10cm) coordinate (B\x);
				\draw[dashed] (bar)--(B\x);
			}
			\filldraw[white] (bar) circle (21.5pt);
			\node at (bar) {$\mathcal{B}_{\mathcal{K}^{\langle 1\rangle}}$};
			\path (bar) ++(0.75,-10) coordinate (W1) node{$ W_1 $};
		\end{tikzpicture}
	\caption{Example of a fractal with $k=10$ not having GLP. Complexes in the closed slice $\overline{W_1}$ cannot be divided into two separate classes according to Theorem \ref{thm:even}.}\label{W1dla10}
	\end{figure}
	
	\begin{theorem}
\label{thm:closedslice}
		If 
		\begin{itemize}
		\item $k>7$ is even \\
		or 
		\item $k=6$ and $ \mathcal{B}_{\mathcal{K}^{\langle 1\rangle}} \notin \mathcal{K}^{\langle 1\rangle} $,
		\end{itemize}
		then $\mathcal{K}^{\langle \infty\rangle}$ has GLP if and only if vertices of complexes in a closed slice $\overline{W_i}$ can be well labeled.
		\end{theorem}
	
	\begin{proof}
		We will show that good labeling of a closed slice $\overline{W_i}$ implies good labeling of a neighbor slice $W_{i+1}$.
		
		At first notice that complexes $\Delta_{0} \subset \overline{W_i}$ adjacent to complexes $\Delta_{0}^{'} \subset W_{i+1}$ must intersect the symmetry axis of $ \mathcal{K}^{\langle 1 \rangle}$. They cannot be tangent to the axis, because it would imply that two complexes meet at two vertices (and have a common edge).
		
		Since $\overline{W_i}$ can be well labeled, we use Theorem \ref{thm:even} to divide complexes $\Delta_{0} \subset \overline{W_i}$ into two separate classes. 
		Then we "reflect" these classes onto $\overline{W_{i+1}}$ by symmetry with respect to axis between $\overline{W_i}$ and $\overline{W_{i+1}}$.
		We have now well labeled $\overline{W_i}\cup \overline{W_{i+1}}$. And since $\overline{W_i}\cup \overline{W_{i+1}}\supset W_i\cup W_{i+1}$, we have also well labeled $ W_i\cup W_{i+1} $.
		We then follow the steps of the proof of Theorem \ref{tw:niep}.
	\end{proof}
	\subsection{GLP for fractals with $k=2^n$}
	We know that if fractal has a prime number of essential fixed points, then it has GLP. In \cite{bib:KOPP} it was also proved that if $k=4$, then the fractal has GLP.
	There was a conjecture, that for $k$ being higher powers of 2 GLP is certain.
	
		\begin{theorem}
		\label{thm:2^n}
		If the fractal has $2^n$ essential fixed points, $n \in \mathbb{N}, n>1$, then it has GLP.
		\end{theorem}
		
		The main idea of the proof is to show that any cycle of adjacent complexes $\Delta_0 \subset \mathcal{K}^{\langle 1 \rangle}$ has even length.\\
		
		\begin{lemma}
		\label{lem:cykle}
		Fractal $\mathcal{K}^{\langle\infty\rangle}$ with even $k$ has GLP if and only if any closed path consisting of vectors connecting centers of adjacent complexes $\Delta_0 \subset \mathcal{K}^{\langle 1 \rangle}$ has even length.
		\end{lemma}
		\begin{proof}
		We consider a graph with vertices in barycenters of complexes $\Delta_0 \subset \mathcal{K}^{\langle 1 \rangle}$. We connect vertices with an edge if respective complexes are adjacent.
		
		By Theorem \ref{thm:even} we know that $\mathcal{K}^{\langle\infty\rangle}$ with even $k$ has GLP if and only if complexes $\Delta_0 \subset \mathcal{K}^{\langle 1 \rangle}$ can be divided into two classes such that complexes from one class are adjacent only to complexes from the other. Considering the given graph it means that the fractal has GLP if and only if that graph is bipartite.

		From general graph theory we know that bipartiteness of the graph is equivalent to the fact that each cycle in that graph has even length. We obtain the thesis.
		\end{proof}
		
		We follow the definitions by \cite{bib:Schroeder} and \cite{bib:Weisstein}.
		
		\begin{definition}
		The minimal polynomial of an algebraic number $x$ is the monic polynomial with rational coefficients of the lowest degree possible such that $x$ is its root.
		\end{definition}
	  For example, for $\sqrt{2}$ one can easily see that the minimal polynomial is $ x^2-2.$ In our proof we will consider numbers of form $e^{2 \pi i p/q}$ for $ p/q\in\mathbb{Q}.$ They are algebraic numbers.
		\begin{definition}
		The primitive root of unity is a number $e^{2 \pi i \frac{l}{n}}$, where $l$ and $n$ are natural and coprime.
		\end{definition}
		\begin{definition}
		The cyclotomic polynomial is a monic polynomial of a form
			$$
			\Phi_{n}(x)=\prod_{1 \leq l \leq n \atop \operatorname{GCD}(l, n)=1}\left(x-e^{2\pi i\frac{l}{n}}\right).
			$$
		\end{definition}
		\begin{proposition}{\cite{bib:Weisstein}}
		Function $\Phi_{n}(x)$ is a polynomial of degree $\phi(n)$ with integer coefficients. Here $\phi(n)$ is an Euler totient function. The function $\Phi_n(x)$ is a minimal polynomial for $n$th primitive roots of unity.
		\end{proposition}
		We may now proceed to the proof of Theorem \ref{thm:2^n}.
		\begin{proof}[Proof of Theorem \ref{thm:2^n}]
		Assume that $k=2^n$ for some $n \in \mathbb{N}, n>1$. We will show that every cycle of complexes $\Delta_0 \subset \mathcal{K}^{\langle 1 \rangle}$ is of even length. By Lemma \ref{lem:cykle} it means that the fractal has GLP.

		We denote the elements of $V_0^{\langle 0 \rangle}$ (i.e. vertices of a complex $\mathcal{K}^{\langle 0 \rangle}$) as $A_0, ..., A_{k-1}$ counter clockwise. In order to simplify further calculations we scale and rotate the complex so that $A_0 = (-1,0)$ oraz $A_{k/2} = (1,0)$. Then the barycenter of the complex is at $(0,0)$ and sides of the polygon $\mathcal{H}\left(\mathcal{K}^{\langle 0 \rangle}\right)$ have length $s_k = 2\sin(\pi /k)$.

		From Fact 3.1 in \cite{bib:KOPP} we know that the vector connecting centers of adjacent complexes has length equal to the longest diagonal of a regular $k$-gon and is paralell to one of these diagonals.
		
		Let us consider the family of vectors corresponding to the longest diagonals of a regular $k$-gon: $e_j = \overrightarrow{A_j A_{j+k/2}}$ for $j\in \{0,...,k-1\}$ (we set $A_{k+j}=A_j$). We consider also the family of vectors $v_j = \left(\cos\left(\frac{2 \pi j}{k}\right), \sin\left(\frac{2 \pi j}{k}\right)\right)$ corresponding to $k$th roots of unity on a complex plane.
		
		We have $\frac{1}{2} e_j = v_j$ for $j =0, ..., k-1$, because $v_j$ are just halves of the longest diagonals of a regular $k$-gon with vertices in roots of unity.
		
		We consider a cycle $C$ of $0$-complexes. The sum of vectors connecting barycenters of subsequent complexes must be equal to zero. Therefore for some $d_j \in \mathbb{N}$ (meaning the quantity of vectors in the cycle) we obtain
		\begin{equation*}
		\sum_{j=0}^{k-1} d_j e_j = 0
		\end{equation*}
		
		Since $e_{j+k/2} = - e_j$ dla $j=0,...,\frac{k}{2}-1$, we can simplify the sum. Let $\widetilde{d}_j = d_j - d_{j+k/2} \in \mathbb{Z}$. Then
		\begin{equation*}
		\sum_{j=0}^{k/2 -1} \widetilde{d}_j e_j= 0
		\end{equation*}
		
Equivalently 

		\begin{equation*}
		\sum_{j=0}^{k/2 -1} \widetilde{d}_j v_j = 0
		\end{equation*}

so

			\begin{equation}
			\label{eq:wielomian}
				\sum_{j=0}^{k/2-1} \widetilde{d}_j \left( e^{2\pi i\frac{1}{k}}\right) ^j = 0.
			\end{equation}

			Clearly $e^{2\pi i\frac{1}{k}}$ is a primitive root of unity. We define a polynomial
			$$
			P(z):=\sum_{j=0}^{k/2-1} \widetilde{d}_j z^j.
			$$
			Then $P\left( e^{2\pi i\frac{1}{k}}\right) =0$. But for $e^{2\pi i\frac{1}{k}}$ a minimal polynomial is of degree $\phi(k) =\phi(2^n)=2^{n-1}$ and $\deg P(z)\leq 2^{n-1}-1<2^{n-1}$.
			
			Hence $P$ must be a zero polynomial, i.e. $\widetilde{d}_j = 0 $ for $j = 0, ..., \frac{k}{2}-1$, what gives
			
			\begin{equation*}
			d_j = d_{j+k/2} \quad \textrm{ for } j = 0, ..., \frac{k}{2}-1.
			\end{equation*}
		
			It means that a number of complexes in cycle $C$ equal to $\sum_{i=0}^{k-1} d_i$ is an even number.
			
			Since every cycle of complexes $\Delta_0 \subset \mathcal{K}^{\langle 1 \rangle}$ has even length, these complexes form a bipartite structure and by Lemma \ref{lem:cykle} the fractal has GLP.
			\end{proof}
		
		One can now ask: are there any other values of $k$ (composite and not powers of 2) such that every fractal with $k$ essential fixed points always has GLP?
		The answer is negative. 
		\begin{theorem}\label{thm:jedynosc}
			For each $k$ being composite and not a power of two we can construct fractal with $k$ essential fixed points not having GLP. 
		\end{theorem}
	At first we prove a simple corollary based on Theorem \ref{thm:odd}.
	\begin{corollary}\label{thm:<k}
		For odd $k$ a fractal cannot have cycles of complexes, which length is odd and less than $k$.
	\end{corollary}
	This corollary may serve as a simple necessary condition used to verify GLP on fractals with odd $k$. It cannot be generalized for cycles of even length (cf. Figure \ref{6cykl}).
	
	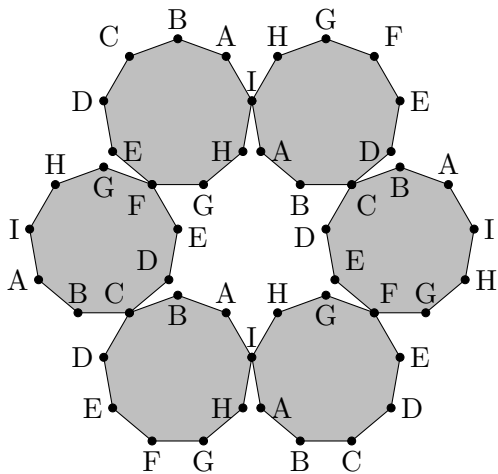
\begin{figure}[H]
		\centering
		\begin{tikzpicture}
			\coordinate (A) at (0,0);	 
			%\path (A) ++(60:10cm) ++(0:2cm)coordinate (B);
			
			\foreach \x in {1,2,...,9}{
				\path (A) ++(\x *40+10:1cm) coordinate (A\x);
			}
			
			\path (A9) ++(-10:1cm) coordinate (B);
			\foreach \x in {1,2,...,9}{
				\path (B) ++(\x *40+10:1cm) coordinate (B\x);
			}
			
			\path (B1) ++(70:1cm) coordinate (C);
			\foreach \x in {1,2,...,9}{
				\path (C) ++(\x *40+10:1cm) coordinate (C\x);
			}
		
			\path (C3) ++(110:1cm) coordinate (D);
			\foreach \x in {1,2,...,9}{
				\path (D) ++(\x *40+10:1cm) coordinate (D\x);
			}
		
			\path (D4) ++(190:1cm) coordinate (E);
			\foreach \x in {1,2,...,9}{
				\path (E) ++(\x *40+10:1cm) coordinate (E\x);
			}
		
			\path (E6) ++(230:1cm) coordinate (F);
			\foreach \x in {1,2,...,9}{
				\path (F) ++(\x *40+10:1cm) coordinate (F\x);
			}
		
			\foreach \x in {A,B,C,D,E,F}{
				\filldraw[fill=lightgray,draw=black] (\x1) -- (\x2) -- (\x3) -- (\x4) -- (\x5) -- (\x6) -- (\x7)-- (\x8)-- (\x9) --cycle;
				\foreach \y in {1,2,...,9}{
					\filldraw [black] (\x\y) circle (1.5pt);
				}
			}
		
			\node[above] at (A1) {A};
			\node[below] at (A2) {B};
			\node[above, xshift=-0.2cm] at (A3) {C};
			\node[left] at (A4) {D};
			\node[left] at (A5) {E};
			\node[below] at (A6) {F};
			\node[below] at (A7) {G};
			\node[left] at (A8) {H};
			\node[above] at (A9) {I};
			
			\node[right] at (B5) {A};
			\node[below] at (B6) {B};
			\node[below] at (B7) {C};
			\node[right] at (B8) {D};
			\node[right] at (B9) {E};
			\node[above, xshift=0.2cm] at (B1) {F};
			\node[below] at (B2) {G};
			\node[above] at (B3) {H};
			
			\node[above] at (C1) {A};
			\node[below] at (C2) {B};
			\node[below, xshift=0.2cm] at (C3) {C};
			\node[left,yshift=-0.1cm] at (C4) {D};
			\node[above right] at (C5) {E};
			\node[above] at (C7) {G};
			\node[right] at (C8) {H};
			\node[right] at (C9) {I};
			
			\node[above right] at (D1) {F};
			\node[above] at (D2) {G};
			\node[above] at (D3) {H};
			\node[above] at (D4) {I};
			\node[right] at (D5) {A};
			\node[below] at (D6) {B};
			\node[left] at (D8) {D};
			\node[right] at (D9) {E};
			
			\node[above] at (E1) {A};
			\node[above] at (E2) {B};
			\node[above left] at (E3) {C};
			\node[left] at (E4) {D};
			\node[right] at (E5) {E};
			\node[below, xshift=-0.2cm] at (E6) {F};
			\node[below] at (E7) {G};
			\node[left] at (E8) {H};
			
			\node[below] at (F2) {G};
			\node[above] at (F3) {H};
			\node[left] at (F4) {I};
			\node[left] at (F5) {A};
			\node[above] at (F6) {B};
			\node[above left] at (F8) {D};
			\node[right,yshift=-0.1cm] at (F9) {E};
		\end{tikzpicture}
	\caption{An example of well labeled six complex cycle for $k=9$.}\label{6cykl}
	\end{figure}
	\begin{proof}
		We know that the number of all rotations in cycle is equal to $c+d$, which is odd and less than $k$.
		The following inequalities hold:
		\begin{equation}
			0<|c-d|<k.
		\end{equation} 
	The left inequality is due to a fact, that we subtract two numbers of different parity. The right one is obvious.
	
	Hence $k\nmid(c-d)$, so the condition (B) from Theorem \ref{thm:odd} is not satisfied.
	\end{proof}
We can notice that Theorem \ref{thm:prime} and Corollary \ref{thm:<k} imply that for prime $k$ there exist no cycles of odd length smaller than $k$.

We now return to the proof of the theorem.
	\begin{proof}[Proof of Theorem \ref{thm:jedynosc}]
	We fix $k$. Our aim is to arrange several $0$-complexes in cycle and show that they cannot be well labeled. 
	
		For a given $0$-complex any of its neighbor is its copy translated by a vector $ e^{2\pi i\frac{j}{k}} $ for some $ j\in\{0,1,...,k-1\} $. 
		We consider two cases.
				
		Case 1. 
		
		$k=2^n r$ for $n\geq 1$ and odd $ r>1 $.\\
 We consider a set of vectors $ E=\{e^{2\pi i\frac{j}{2^n r}}: j=2^n h, h\in\{0,1,...,r-1\}  \} $. Notice that 
		$$ 
		\sum_{h=0}^{r-1}e^{2\pi i\frac{2^n h}{2^n r}}= \sum_{h=0}^{r-1}e^{2\pi i\frac{h}{r}}=0,
		$$
		because we sum all $r$th roots of unity.\\
		Now we produce a cycle of length $r$. The barycenter of $q$-th complex $\Delta_{0}$ for $q \in \{1,...,r\}$ will be placed at
		$$
		\sum_{h=0}^{q-1}e^{2\pi i\frac{h}{r}}.
		$$
		Since this cycle has odd number of elements, due to Theorem \ref{thm:even} it cannot be well labeled.
		
		Case 2. 
		
		$ k=nr$ for any odd $ n>1$ and $ r>1 $.\\
		As previously, we cosider a set $ E=\{e^{2\pi i\frac{j}{nr}}: j=hn, h\in\{0,1,...,r-1\}  \} $. We may notice that again
		$$ 
		\sum_{h=0}^{r-1}e^{2\pi i\frac{h n}{n r}}= \sum_{h=0}^{r-1}e^{2\pi i\frac{h}{r}}=0.
		$$
		We produce a cycle of length $r$. The barycenter of $q$-th complex $\Delta_{0}$ for $q \in \{1,...,r\}$ will be placed at
		$$
		\sum_{h=0}^{q-1}e^{2\pi i\frac{h}{r}}.
		$$
		Since a number of rotations in this cycle is equal to $r<k$, due to Corollary \ref{thm:<k} it cannot be well labeled.
		
		In both cases we can take $k$ neighboring cycles (forming one big cycle) to construct a $1$-complex $\mathcal{K}^{\langle 1 \rangle}$. In some cases it may be needed to add some complexes to the structure to separate these cycles and assure that they do not cover any other. Anyway, such structure cannot be well labeled, therefore a fractal does not have GLP.
		\end{proof}

For every $k\geq 3$ one can of course construct an example of a nested fractal having GLP. If $4\nmid k$ it is enough to take $N=k$, so the fractal $\mathcal{K}^{\langle 1 \rangle}$ would consist of $k$ 0-complexes creating one cycle (as in Sierpi\'nski hexagon).

If $4|k$, then in such construction neighboring complexes would share whole edges, therefore the set would not be a nested fractal. We can take $N=2k$. We place $k$ 0-complexes in vertices of $\mathcal{K}^{\langle 1 \rangle}$ and additionally between each pair of complexes in neighboring vertices we place another $0$-complex in such way, that it shares one vertex with both adjacent complexes.

Summing up, for each $k\geq 3$ there exist fractals having GLP. If $k$ is prime or $k=2^n$, then fractal surely has GLP. If $k$ is composite and has nontrivial odd divisor, then it is possible that it does not have that property.

\bigskip
% {
%\noindent
%\textbf{Acknowledgements.} We would like to thank the anonymous referee for his valuable suggestions and comments.
%}

\end{document}